\documentclass{article}
\usepackage{amssymb}
\usepackage{eurosym}
\usepackage{epsfig}
\usepackage{amsmath}
\usepackage{amsfonts}
\usepackage{pgf,tikz}
\usepackage{enumerate}
\usepackage{mathrsfs}
\usepackage{cite}
\usepackage{subfigure}
\usepackage[colorlinks=true, linkcolor=black, citecolor=black, urlcolor=black]{hyperref}
\usepackage[top=2cm, bottom=2.5cm, left=2.5cm, right=2.5cm, headsep=1cm]{geometry}

\newtheorem{theorem}{\sc Theorem}[section]
\newtheorem{proposition}{\sc Proposition}[section]

\newtheorem{definition}{\sc Definition}[section]
\newtheorem{remark}{\sc Remark}[section]
\newtheorem{corollary}{\sc Corollary}[section]
\newtheorem{example}{\sc Example}[section]

\def\qed{\hbox to 0pt{}\hfill$\rlap{$\sqcap$}\sqcup$\medbreak}

\title{A unified approach to compression--expansion fixed point theorems for operators systems and applications}

\author{Laura Mª Fern\'andez--Pardo$\,^2$ and Jorge Rodr\'iguez--L\'opez$\,^{1,2}$} 

\date{}
\begin{document}
 \maketitle

\begin{center}  {\small $^1$ CITMAga, 15782, Santiago de Compostela, Spain.  \\
		 $^2$ Departamento de Estat\'{\i}stica, An\'alise Matem\'atica e Optimizaci\'on, \\ Universidade de Santiago de Compostela, \\ 15782, Facultade de Matem\'aticas, Campus Vida, Santiago, Spain.\\  Email: laura.fernandez.pardo@usc.es; jorgerodriguez.lopez@usc.es}
\end{center}

\medbreak

\noindent {\it Abstract.} In this paper, we present some fixed point theorems for operator systems in the line of Krasnosel'ski\u{\i}'s theorem in cones. The cone-compression and cone-expansion type conditions are imposed in a component-wise manner. Unlike related results in the literature, the operators are allowed to be defined in the Cartesian product of conical regions delimited by nonconvex sets. Our approach, based on the fixed point index, ensures the existence of a coexistence fixed point—that is, one with nontrivial components. As a first application, we establish several localization results for systems of integral equations between strictly star-shaped sets defined by functionals. These results cannot be derived solely from previous studies dealing with operators in annular regions. A second application concerns nonlinear systems involving the $\Phi$-Laplacian.

\medbreak


\medbreak


\section{Introduction}

One of the principal tools in Nonlinear Analysis for proving the existence of nontrivial solutions to boundary value problems is Krasnosel’ski\u{\i}’s compression–expansion fixed point theorem \cite{guolak,Kras}. When dealing with systems of operator equations of the form
\begin{equation*}
	\left\{\begin{array}{l} x_1=T_1(x_1,x_2), \\ x_2=T_2(x_1,x_2), \end{array} \right.
\end{equation*}
the fixed points obtained through Krasnosel’ski\u{\i}’s theorem are not localized independently for each component, leaving open the possibility that some of them may be trivial. To address this issue, several \textit{vector versions of Krasnosel’ski\u{\i}’s fixed point theorem} have been formulated \cite{LFP_JRL,PrecupMAA,PrecupFPT,JRL}, imposing compression–expansion conditions on each component of the compact operator $T=(T_1,T_2)$ and guaranteeing a fixed point with all components nontrivial, i.e., a \textit{coexistence} fixed point. In \cite{PrecupFPT} the conditions are stated in the classical way, in terms of the partial order induced by the respective cones $K_1 \times K_2\subset X_1\times X_2$ in which the operator acts. Meanwhile, in \cite{JRL,PrecupSDC}, the so-called homotopic conditions of Krasnosel’ski\u{\i}’s theorem are imposed, whereas in \cite{LFP_JRL} the conditions are expressed in terms of the respective norms of the normed spaces forming the product $X_1 \times X_2$.

In the classical setting, one can impose Krasnosel’ski\u{\i}’s compression–expansion conditions on the operator over the boundary of two bounded, relatively open sets $U$ and $\mathcal{O}$ of the cone, with $0 \in U$ and $\overline{U} \subset \mathcal{O}$, when the operator acts on a Banach space $X$. In \cite{guolak}, this is established via the fixed point index, which is well-defined only when the operator has no fixed points on the boundary of its domain. The reasoning relies on the Dugundji extension theorem for compact operators, allowing the extension of the compact operator from $\overline{\mathcal{O}}\setminus U$ to $\overline{\mathcal{O}}$, while preserving the index defined on $U$, $\mathcal{O}$, and $\mathcal{O}\setminus\overline{U}$.

In product spaces, assuming that each $X_j$ is a Banach space, $U_j$ and $\mathcal{O}_j$ are bounded relatively open subsets of $K_j$ with $0 \in U_j$ and $\overline{U}_j \subset \mathcal{O}_j$ ($j=1,2$),  and applying the Dugundji theorem does not suffice to extend an operator from $(\overline{\mathcal{O}}_1 \setminus U_1)\times(\overline{\mathcal{O}}_2 \setminus U_2)$ to $\overline{\mathcal{O}}_1\times\overline{\mathcal{O}}_2$ while keeping the fixed point index well-defined on the relevant subsets for the final computation. This requirement forces the results in \cite{JRL,LFP_JRL} to be essentially restricted to operators defined on sets of the form
\[\overline{K}_{r,R}:=\{x=(x_1,x_2)\in K_1\times K_2:r_j\leq\|x_j\|_{X_j}\leq R_j \text{ for }j=1,2\},\]
with $0<r_j<R_j\,(j=1,2)$. That is, on the product of annular shells of the cones. 

In this paper, we overcome this difficulty by stating a new and more general result, which applies to operators defined on a much more general domain. Namely, in those of the form 
\[(\overline{\mathcal{O}}_1 \setminus \Omega_1)\times(\overline{\mathcal{O}}_2 \setminus \Omega_2),\]
with $\mathcal{O}_j$ a bounded relatively open set of $K_j$ and $\Omega_j$ a \textit{strictly star-shaped} set such that $\overline{\Omega}_j\subset\mathcal{O}_j$ $(j=1,2)$.
Notice that such type of domain is general enough to cover most of applications involving systems of differential or integral equations.

The class of strictly star-shaped sets had been already considered successfully in fixed point theory. For instance, a Rothe type fixed point theorem in strictly star-shaped domains was proposed by Deimling in \cite[Ex. 5, p. 33]{Deim} (see also the proof due to Zanolin \cite[Corollary 1]{z}). This class of sets was also mentioned by Kwong \cite{kwongNA} in relation with Brouwer fixed point theorem. In particular, it is stated that strictly star-shaped sets are homeomorphic to the closed unit ball. A precise proof of this fact, in the finite dimensional setting, can be found in \cite[Lemma 3.2]{fz}. We also notice that \textit{strictly star shaped} sets are precisely the class of star convex sets considered in \cite{LR,LPR} in the context of Krasnosel'ski\u{\i} compression--expansion fixed point theorem.

Strictly star-shaped sets are also widely used in the study of solutions to integral and differential problems. For instance, in \cite{fz}, Feltrin and Zanolin employed strictly star-shaped sets in the study of periodic solutions of first order systems. Existence results involving this type of sets, defined by concave and convex functionals, are also established in \cite{infante, eloe}.

The fixed point theory developed here has a two-fold interest when compared with the related literature:
\begin{enumerate}
	\item[$i)$] it provides a unified approach to the cone-compression and cone-expansion conditions in case of systems, covering the usual homotopy type and normed type conditions, see Remark~\ref{remark_versions} below.
	\item[$ii)$] it applies to operators defined in general domains of the form $(\overline{\mathcal{O}}_1 \setminus \Omega_1)\times(\overline{\mathcal{O}}_2 \setminus \Omega_2)$, as explained before.   
\end{enumerate}

Relying on our main result, we establish several localization results for a Hammerstein type system of integral equations. Following the ideas due to Webb \cite{web}, we impose on the nonlinearities weaker growth conditions than those commonly used in similar existence results (see, for instance, \cite{Figue_tojo,GKM,infante, lan2001} for results in the classical case and \cite{Lan1,JRL} for the vectorial approach), so that a nonlinearity may take larger values on part of its domain as long as it remains sufficiently small on another.
A second application concerns systems of $\Phi$-Laplacian equations with mixed boundary conditions. Our results encompass singular, classical, and component-wise combinations of singular and classical $\Phi$-Laplacians. The resulting localization—derived using both the $L^1$-norm and the supremum norm—ensures the existence of a positive coexistence solution. Complementary results for differential problems involving the $\Phi$-Laplacian, obtained through topological methods, can be found, for instance, in \cite{bereanu,herlea,rachu}.

The paper is organized as follows. In Section 2, we introduce strictly star-shaped sets and construct a retraction essential for the proofs of the main results. In Section 3, we recall the basic properties of the Leray-Schauder fixed point index and establish new fixed point theorems in Cartesian products of cones. In Section 4, we introduce conditions on functionals defining strictly star-shaped sets, illustrating the flexibility of the abstract theory for the localization of nontrivial solutions. Finally, Section 5 is devoted to the localization of coexistence positive solutions for both Hammerstein-type systems of integral equations and $\Phi$-Laplacian systems.

\section{Star-shaped sets}

In the sequel, let $(X,\left\|\cdot\right\|)$ be a normed space.

\begin{definition}
	A subset $A\subset X$ is said to be a $p$-\textit{star-shaped} or $p$-\textit{star convex set} if 
	\[(1-\lambda)\,p+\lambda\,x\in A \quad \text{for all } \lambda\in[0,1] \text{ and all } x\in A, \]
	that is, the set $A$ contains the whole segment connecting $p$ and every $x\in A$.  
	
	If $p=0$, $A$ is simply said to be a \textit{star-shaped} or a \textit{star convex set}.
\end{definition}

We now restrict our attention to a particular class of these sets on which we will focus from now on, the strictly star-shaped sets. In particular, we will follow the terminology employed in \cite{fz} in the finite dimensional setting.

\begin{definition}
 Let $\Omega\subset E$ be a nonempty relatively open bounded subset of $E$, where $E\subset X$. We will say that $\overline{\Omega}$ is \textit{strictly star-shaped} with respect to a point $p$ over $E$ if
 \[(1-\lambda)\,p+\lambda\,x\in \Omega \quad \text{for all } \lambda\in [0,1) \text{ and all } x\in \partial_E\,\Omega, \]
 where $\partial_E\,\Omega$	denote the relative boundary of $\Omega$ in $E$.
 
 If $p=0$, we will say that $\overline{\Omega}$ is a \textit{strictly star-shaped set over $E$}. Moreover, if it is over the whole space, we will simply say that it is a strictly star-shaped set.
\end{definition}

Observe that, in particular, the previous definition implies that the point $p$ belongs to the set $\Omega$. Moreover, strictly star-shaped sets are those star-shaped sets whose boundary contains no segments aligned with $p$, as noted in \cite{Jmelado_Hmoral}, where an alternative and equivalent definition is given, alongside the proof that if $\overline{\Omega}$ is convex, then it is strictly star-shaped with respect to any point $p\in\Omega$.

Note that \textit{strictly star-shaped} sets are precisely the class of star convex sets considered in \cite{LPR}. Indeed, strictly star-shaped sets with respect to $p$ are those sets lying between star convex sets that satisfy condition (2.1) in \cite{LPR}, that is, 
\begin{equation}\label{cond_LPR}
	\text{for all } x\in \overline{\Omega}\setminus\{p\}, \text{ there is a unique } \lambda_x>0 \text{ such that } (1-\lambda_x)\,p+\lambda_x\,x\in \partial\,\Omega.
\end{equation} 

\begin{proposition}
 Let $\Omega$ be a nonempty open bounded subset of $X$. The set $\overline{\Omega}$ is a strictly star-shaped set with respect to a point $p$ if, and only if, it is a $p$-star convex set and satisfies condition \eqref{cond_LPR}.
\end{proposition}

\noindent
{\bf Proof.} Suppose that $\overline{\Omega}$ is a strictly star-shaped set with  respect to $p$. Let us prove that it fulfills condition \eqref{cond_LPR}. Let $x\in \overline{\Omega}\setminus\{p\}$ be fixed. First, let us show the existence of the positive number $\lambda_x$. We already know that $\overline{\Omega}$ is a $p$-star convex set, so we have that $(1-\lambda)\,p+\lambda\,x\in\overline{\Omega}$ for all $\lambda\in[0,1]$. On the other hand, since $\overline{\Omega}$ is bounded, there exists $\lambda>1$ large enough such that $(1-\lambda)\,p+\lambda\,x\notin\overline{\Omega}$. Then it suffices to choose \[\lambda_x=\inf\{\lambda>1:(1-\lambda)\,p+\lambda\,x\notin\overline{\Omega} \},\]
in order to obtain that $(1-\lambda_x)\,p+\lambda_x\,x\in \partial\,\Omega$.

Now, let us prove uniqueness. Suppose that $\lambda_x^1,\lambda_x^2>0$ satisfy $(1-\lambda_x^i)\,p+\lambda_x^i\,x\in \partial\,\Omega$, $i=1,2$, with $\lambda_x^1\neq \lambda_x^2$. Let us assume that $\lambda_x^1<\lambda_x^2$ and denote $\lambda=\lambda_x^1/ \lambda_x^2$. Hence, we have that $(1-\lambda_x^2)\,p+\lambda_x^2\,x\in \partial\,\Omega$ and $\lambda\in(0,1)$, so the fact that $\overline{\Omega}$ is a strictly star-shaped set with respect to $p$ ensures that
\[(1-\lambda)\,p+\lambda\left[(1-\lambda_x^2)\,p+\lambda_x^2\,x \right]\in\Omega \]
or, equivalently, $(1-\lambda_x^1)\,p+\lambda_x^1\,x\in \Omega$. Finally, since $\Omega$ is an relatively open set, we reach a contradiction with $(1-\lambda_x^1)\,p+\lambda_x^1\,x\in \partial\,\Omega$.

Suppose now that $\overline{\Omega}$ is $p$-star convex and satisfies condition \eqref{cond_LPR}. Let $x\in\partial\,\Omega$, as $\overline{\Omega}$ is a $p$-star convex set it, follows that $(1-\lambda)\,p+\lambda \,x\in \overline{\Omega}$ for all $\lambda\in[0,1]$. From \eqref{cond_LPR} we know that only for $\lambda=1$ we have that $(1-\lambda)\,p+\lambda\, x\in \partial\Omega$. Then $(1-\lambda)\,p+\lambda\, x\in \Omega$ for $\lambda\in[0,1)$.
\qed

\begin{remark}
	\label{rem_2.1}
	Let $E$ be a convex set in $X$ and $\overline{\Omega}$ a strictly star-shaped set with respect to a point $p\in E$. It is clear that $\overline{\Omega}^\star:=\overline{\Omega}\cap E$ is a strictly star-shaped set with respect to $p$ over $E$ and fulfills the following slightly modified version of condition \eqref{cond_LPR}
	\begin{equation}\label{cond_LPR2}
		\text{for all } x\in \overline{\Omega}^\star\setminus\{p\}, \text{ there is a unique } \lambda_x>0 \text{ such that } (1-\lambda_x)\,p+\lambda_x\,x\in \partial_E\,\Omega.
	\end{equation} 
\end{remark}

Let us now recall the concept of a cone in a normed space. A closed and convex subset $K$ of a normed linear space $(X,\left\|\cdot\right\|)$ is a \textit{wedge} if $\lambda\,u\in K$ for every $u\in K$ and for all $\lambda\geq 0$. Furthermore, a wedge $K$ is said to be a \textit{cone} if, in addition, it satisfies that $K\cap(-K)=\{0\}$. 

Next, we consider strictly star-shaped sets over a cone. Let $K$ be a cone in the normed space $X$ and let $\Omega\subset K$ be a nonempty relatively open bounded set such that $\overline{\Omega}$ is strictly star-shaped (over the cone $K$). Since strictly star-shaped sets in a cone satisfy condition \eqref{cond_LPR2} with $p=0$, the following result can be derived in a straightforward manner from Remark \ref{rem_2.1} together with \cite[Theorem 2.2]{LPR} and  \cite[Proposition 2.4]{LR}.

\begin{proposition}
	For every $x\in\overline{\Omega}\setminus\{0\}$ there exists a unique number $\beta_x\in[1,+\infty)$ such that $\beta_x\,x\in \partial_K\,\Omega$. Moreover, the mapping 
	\[\beta:\overline{\Omega}\setminus\{0\}\rightarrow[1,+\infty), \qquad \beta(x):=\beta_x \]	
	is continuous and $\beta(x)\to +\infty$ as $x\to 0$. 
\end{proposition} 

%\begin{remark}
%	Observe that the map $\beta$ can be continuously extended to $K\setminus\{0\}$ as
%	\[\beta^K:K\setminus\{0\}\to (0,+\infty), \quad \beta^K(u)=\beta\left(\dfrac{d(0,\partial_K\,\Omega)}{\left\|u\right\|}u\right)\dfrac{d(0,\partial_K\,\Omega)}{\left\|u\right\|}. \]
%	First, note that being $\overline{\Omega}^K$ a strictly star shaped set it satisfies that $0\in\Omega$ and thus, since $\Omega$ is relatively open, it follows that $d(0,\partial_K\,\Omega)>0$. In addition, it is clear that if $u\in K\setminus\{0\}$, then
%	\[v:=\dfrac{d(0,\partial_K\,\Omega)}{\left\|u\right\|}u\in \overline{\Omega}^K\setminus\{0\}. \] 
%	Indeed, $v=\mu\,u$, with $\mu=d(0,\partial_K\,\Omega)/\left\|u\right\|>0$, so the definition of cone implies that $v\in K\setminus\{0\}$. On the other hand, the closed ball $\overline{B(0,r)}$, with radius $r:=d(0,\partial_K\,\Omega)$, satisfies that $K\cap\overline{B(0,r)}\subset \overline{\Omega}^K$ and thus $v\in\overline{\Omega}^K$. Hence, the map $\beta^K$ is well-defined and, moreover, it is continuous since so is $\beta$.
%	
%	Furthermore, for every $u\in K\setminus\{0\}$, the map $\beta^K$ gives a positive number $\beta^K(u)>0$ such that \[\beta^K(u)\,u\in \partial_K\,\Omega,\]
%	since $\beta^K(u)\,u=\beta(v)\,v:=\beta_v\,v\in \partial_K\,\Omega$. Therefore, by uniqueness, we have
%	$\beta^K(u)=\beta(u)$ provided that $u\in \overline{\Omega}^K\setminus\{0\}$. 
%\end{remark}

Using the map $\beta$, we now show that if $\overline{\Omega}$ is a strictly star-shaped set over a cone $K$, then $\partial_K\,\Omega$ is a retract of $\overline{\Omega}$. The construction of the retraction is inspired by \cite[Example 3]{fel}. 

This result will be crucial in the next section in order to establish our fixed point theorems for operator systems in strictly star-shaped sets.

\begin{proposition}
	\label{prop_retract}
	Let $K$ be a cone in the normed space $X$ and  $\Omega\subset K$ a relatively open bounded set such that $\overline{\Omega}$ is strictly star-shaped over $K$. Let  $r\in\mathbb{R}_+:=(0,+\infty)$ be such that $\overline{K}_r:=\left\{x\in K:\left\|x\right\|\leq r \right\}\subset \Omega$. For $h\in K\setminus\{0\}$ fixed, the map $\rho^r_h:\overline{\Omega}\rightarrow \partial_K\,\Omega$ defined as
	\begin{equation}
		\label{equa_retra}
\rho^r_h(x):=\left\{\begin{array}{ll} \beta\left(r\frac{x+(r-\|x\|)h}{\|x+(r-\|x\|)h\|}\right)\left(r\frac{x+(r-\|x\|)h}{\|x+(r-\|x\|)h\|}\right), & \text{ if } x\in \overline{K}_r, \\ \beta(x)\,x, & \text{ if } x\in \overline{\Omega}\backslash\overline{K}_r, \end{array} \right.
	\end{equation}
 is a retraction. 
\end{proposition}

\noindent
{\bf Proof.} First of all, notice that $\|x+(r-\|x\|)h\|\neq 0$ for all $x\in \overline{K}_r$. Indeed, if not we have $-x=(r-\|x\|)h\in K$ which jointly with $x\in K$ and that $K$ is a cone, imply $x=0$. Consequently, $\|rh\|>0$ since $r>0$ and $h\in K\backslash\{0\}$ and $\rho^r_h$ is well-defined.

Clearly, $\rho^r_h$ is continuous in $\overline{\Omega}$ due to the continuity of the function $\beta$.

Finally, if $x\in \partial_K\,\Omega$, then the definition of $\beta$ gives $\beta(x)=1$ and so $\rho^r_h(x)=\beta(x)\,x=x$, that is, the restriction of $\rho^r_h$ to the set $\partial_K\,\Omega$ is the identity map. Therefore, $\rho^r_h$ is a retraction.
\qed

\begin{figure}[h]
	\centering
	\subfigure[Case $x\in K_r$.]{
	\includegraphics[scale=1]{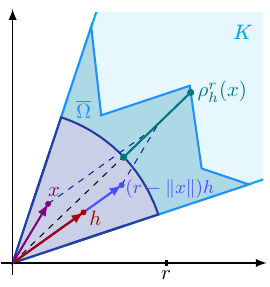}}
	\hspace{2cm}
	\subfigure[Case $x\in\overline{\Omega}\backslash K_r$.]{\includegraphics[scale=1]{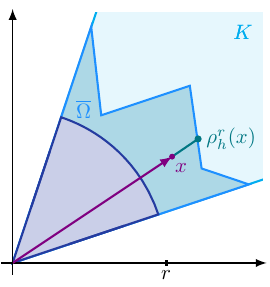}
	}
	\caption{Illustration of the retraction $\rho^r_h$ on $\overline{\Omega}$, a strictly star convex set over a cone $K$ in $\mathbb{R}^2$.}
\end{figure}

%\begin{remark}
%Observe that it is immediate from the expression of the retraction that $\|\rho^r_h(x)\|>\|x\|$ for all $x\in\Omega$.
%\end{remark}

\begin{remark}
	Note also that if $K$ is a wedge instead of a cone, then $\partial_K\,\Omega$ may not be a retract of $\overline{\Omega}$. For instance, it is well-known, due to the classical Borsuk no-retraction theorem \cite[Theorem 7.2]{GraDug}, that there is no retraction from $\partial\,B$ into $B$, with $B$ the closed unit ball in $\mathbb{R}^n$.
\end{remark}

\section{Fixed point theorems for operator systems}

Let $(X_1,\left\|\cdot\right\|_{1})$ and $(X_2,\left\|\cdot\right\|_{2})$ be normed linear spaces and $K_1\subset X_1$ and $K_2\subset X_2$ two cones. When no confusion may occur, both norms $\left\|\cdot\right\|_{1}$ and $\left\|\cdot\right\|_{2}$ will be simply denoted as $\left\|\cdot\right\|$. Moreover, we will employ the notation $K:=K_1\times K_2$ for the corresponding cone in the cartesian product $X:=X_1\times X_2$.

\subsection{Fixed point index computations}

For completeness, we briefly recall some notions and properties concerning the fixed point index in cones. If $U$ is a relatively open bounded subset of $K$ and $T:\overline{U}\rightarrow K$ is a compact map without fixed points on $\partial_K\,U$, the \textit{fixed point index} of $T$ on $U$ with respect to the cone $K$, $i_K(T,U)$, is well-defined.

	\begin{proposition}\label{prop_index}
	Let $K$ be a cone of a normed space, $U\subset K$ be a bounded relatively open set and $N:\overline{U}\rightarrow K$ be a compact map such that $N$ has no fixed points on $\partial_{K}\,U$. Then the fixed point index of $N$ on the set $U$ with respect to $K$, $i_{K}(N,U)$, has the following properties:
	\begin{enumerate}
		\item (Additivity) Let $U$ be the disjoint union of two open sets $U_1$ and $U_2$. If $0\not\in(I-N)(\overline{U}\setminus(U_1\cup U_2))$, then \[i_{K}(N,U)=i_{K}(N,U_1)+i_{K}(N,U_2).\]
		\item (Existence) If $i_{K}(N,U)\neq 0$, then there exists $x\in U$ such that $Nx=x$.
		\item (Homotopy invariance) If $H:\overline{U}\times[0,1]\rightarrow K$ is a compact homotopy and $0\not\in(I-H)(\partial\,U\times[0,1])$, then
		\[i_{K}(H(\cdot,0),U)=i_{K}(H(\cdot,1),U).\]
		\item (Normalization) If $N$ is a constant map with $Nx=\bar{x}$ for every $x\in\overline{U}$, then
		\[i_{K}(N,U)=\left\{\begin{array}{ll} 1, & \text{ if } \bar{x}\in U, \\ 0, & \text{ if } \bar{x}\not\in\overline{U}. \end{array} \right. \]
		%		\item (Multiplicativity) For $j=1,2$, let $C_j$ be a wedge, $U_j\subset C_j$ be a open bounded set and $S_j:U_j\rightarrow C_j$ be a compact map fixed point free on the boundary of $U_j$. Then
		%		\[i_{C_1\times C_2}(S_1\times S_2,U_1\times U_2)=i_{C_1}(S_1,U_1)\cdot i_{C_2}(S_2,U_2). \]
	\end{enumerate}
\end{proposition}

We now present some computations of the fixed point index that will be useful in our arguments. Proofs can be found, for instance, in \cite[Lemma 2.3.1 and 2.3.2]{guolak}.

\begin{proposition}\label{prop_ind01}
	Let $K$ be a cone, $U\subset K$ be a bounded relatively open set such that $0\in U$ and $N:\overline{U}\rightarrow K$ be a compact map without fixed points on $\partial_{K}\,U$.
	\begin{enumerate}[$(a)$]
		\item If $Nx\neq \lambda\, x$ for all $x\in \partial_{K}\, U$ and all $\lambda> 1$, then $i_{K}(N,U)=1$.
		\item If there exists a compact map $L:\overline{U}\rightarrow K$ such that 
		\[\inf_{x\in \overline{U}}\left\| L x\right\|>0, \text{ and} \]
		\[x-N x\neq \mu\, L x \ \text{ for all } x\in \partial_{K}\, U \text{ and every } \mu> 0 ,\] 
		then $i_{K}(N,U)=0$.
	\end{enumerate}	
\end{proposition}

Further details on the fixed point index and the preceding computations can be found in \cite{amann,GraDug,guolak}.

Proposition \ref{prop_ind01} enables the proof of Krasnosel'ski\u{\i}’s compression-expansion fixed point theorem in the classical setting \cite{guolak}. For operators defined on the Cartesian product of normed spaces, as in our case, the computation of the fixed point index under component-wise combinations of the conditions in Proposition~\ref{prop_ind01} was established in \cite[Proposition 2.3]{LFP_JRL}. This result, which we recall below, is also essential for the proof of our main results.

	\begin{proposition}\label{prop_ind_sys}
	Let $U\times V$ be a bounded relatively open subset of the cone $K= K_1\times K_2$ in the normed spaces product $X=X_1\times X_2$, such that $0\in U$. Assume that $N=(N_1,N_2):\overline{U\times V}\rightarrow K$ and $L:\overline{U\times V}\rightarrow K_2$ are compact mappings satisfying the following conditions:
	\begin{enumerate}[$(a)$]
		\item $N_1 x \neq \lambda\, x_1$ for all $x_1\in \partial_{K_1}U$, $x_2\in\overline{V}$ and all $\lambda>1$;
		\item \begin{enumerate}[$(i)$]
			\item $\inf_{x\in \overline{U\times V}}\left\|L x\right\|>0$;
			\item $x_2-N_2 x\neq \mu\, L x$ for all $x_1\in\overline{U}$, $x_2\in\partial_{K_2} V$ and every $\mu>0$.
		\end{enumerate}
	\end{enumerate}	
	If $N$ has no fixed points on $\partial_K\, (U\times V)$, then $i_{K}(N,U\times V)=0$.
\end{proposition}

	\begin{remark}
		It is clear that the conclusion of Proposition \ref{prop_ind_sys} remains valid if the roles of $N_1$ and $N_2$ expressed by assumptions $(a)$ and $(b)$ are interchanged. 
	\end{remark}
	
	\subsection{Main results}
 For each $j\in\{1,2\}$, let $\Omega_j$ and $\mathcal{O}_j$ be relatively open bounded subsets of the cone $K_j\subset X_j$ such that 
\begin{enumerate}[(I)]
	\item $0\in\Omega_j\subset\overline{\Omega}_j\subset\mathcal{O}_j$; and
	\item $\overline{\Omega}_j$ is a strictly star-shaped set over $K_j$.
%	\item $K_i\cap \mathcal{O}_i$ is bounded, with $K_i$ a cone of $X_i$.
\end{enumerate}

Below, we present the main result of our paper, which, as we will see immediately afterwards, encompasses several well-known and widely used formulations of Krasnosel'ski\u{\i}’s theorem, generalizing them to domains more general than those previously considered in the literature \cite{PrecupFPT,PrecupSDC,JRL,LFP_JRL} within the context of product spaces.

\begin{theorem}
	\label{th_ppal}
	Assume that $T=(T_1,T_2):D:=\left(\overline{\mathcal{O}}_1\setminus\Omega_1\right) \times \left(\overline{\mathcal{O}}_2\setminus\Omega_2\right) \rightarrow K$ and $S=(S_1,S_2):D \rightarrow K$ are compact maps such that, for each $j\in\{1,2\}$, one of the following conditions holds in $D$:
	\begin{enumerate}[(A)] 
		\item
		\begin{enumerate}[(a)] 
			\item  $T_jx\neq \lambda x_j$ for $x_j\in\partial_{K_j}\mathcal{O}_j$ and $\lambda\geq 1$,
			\item 	
			\begin{enumerate}[(i)] 
				\item $\inf_{x\in D}\|S_jx\|>0$;
				\item $x_j-T_j x \neq \mu S_j x$ for $x_j\in\partial_{K_j}\Omega_j$ and $\mu \geq 0$.
			\end{enumerate}
		\end{enumerate}
		\item 
		\begin{enumerate}[(a)] 
		\item  $T_jx\neq \lambda x_j$ for $x_j\in\partial_{K_j}\Omega_j$ and $\lambda\geq 1$,
		\item 	
		\begin{enumerate}[(i)] 
			\item $\inf_{x\in D}\|S_jx\|>0$;
			\item $x_j-T_j x \neq \mu S_j x$ for $x_j\in\partial_{K_j}\mathcal{O}_j$ and $\mu \geq 0$.
		\end{enumerate}
	\end{enumerate}
	\end{enumerate}
	
	Then $i_K\left(T,\left(\mathcal{O}_1\setminus\overline{\Omega}_1\right)\times \left(\mathcal{O}_2\setminus\overline{\Omega}_2\right)\right)=(-1)^k$, where $k\in\{0,1,2\}$ is a counter of the number of indexes $j$ for which condition (B) is satisfied.
	
	In particular, the operator $T$ has at least one fixed point in $\left(\mathcal{O}_1\setminus\overline{\Omega}_1\right)\times \left(\mathcal{O}_2\setminus\overline{\Omega}_2\right)$.
\end{theorem}

\noindent
{\bf Proof.} Suppose that both $T_1$ and $T_2$ satisfy condition (A). For each $j \in \{1,2\}$, since $\overline{\Omega}_j$ is a strictly star-shaped set over $K_j$, there exists $r_j \in \mathbb{R}_+$, as given in Proposition~\ref{prop_retract} such that for a fixed $h_j \in K_j \setminus \{0\}$ the map $\rho_j \equiv \rho^{r_j}_{h_j}$ defined by (\ref{equa_retra}) is a retraction. We employ these retractions to extend the operators $T$ and $S$ (in such a way that the extensions remain compact) to the set $\overline{\mathcal{O}_1\times\mathcal{O}_2}$, as follows. Let $N=(N_1,N_2):\overline{\mathcal{O}_1\times\mathcal{O}_2}\rightarrow K$ and $L:\overline{\mathcal{O}_1\times\mathcal{O}_2}\rightarrow K$ denote de extension of $T$ and $S$, respectively, defined by
\[N(x):=T\circ\theta(x)  \text{ and } L(x):=T\circ\theta(x),\]
with $\theta:\overline{\mathcal{O}_1\times\mathcal{O}_2}\longrightarrow D$ given by $\theta(x_1,x_2)=(\theta_1(x_1),\theta_2(x_2))$ where, for each $j\in\{1,2\}$ the map $\theta_j:\overline{\mathcal{O}}_j\longrightarrow \overline{\mathcal{O}}_j\setminus \Omega_j$ has the following form
\begin{equation*}
	\theta_j(x_j) = 
	\begin{cases}
		\rho_j(x_j), & \text{if }  x_j \in \Omega_j, \\
		x_j, & \text{if } x_j\in \overline{\mathcal{O}}_j\setminus \Omega_j.
	\end{cases}
\end{equation*}

The extension $N$ satisfies in $\overline{\mathcal{O}_1\times\mathcal{O}_2}$, for each $j \in \{1,2\}$, the following conditions

\begin{enumerate}[($a^\star$)]
	\item \label{item_a} $N_jx\neq \lambda\, x_j$ for $x_j\in\partial_{K_j}\mathcal{O}_j$ and $\lambda\geq 1$,
	\item  \label{item_b}
	\begin{enumerate}[($i$)]
		\item $\inf_{x\in\overline{\mathcal{O}_1\times\mathcal{O}_2}}\|L_jx\|>0$;
		\item $x_j-N_j x \neq \mu\, L_j x$ for $ x_j\in\partial_{K_j}\Omega_j$ and $\mu \geq 0$.
	\end{enumerate}
\end{enumerate}

Let us check it for $j=1$, the other case is analogous. To verify ($a^\star$), take $x=(x_1,x_2)\in\overline{\mathcal{O}_1\times\mathcal{O}_2}$ with $x_1\in\partial_{K_1}\mathcal{O}_1$. From the definition of $N$, we have $N_1x=T_1(x_1,\rho_2(x_2))$. Observe also that $(x_1,\rho_2(x_2))\in D$ and $x_1\in\partial_{K_1}\mathcal{O}_1$. Then condition (A)-($a$) gives $N_1x\neq\lambda\, x_1$ for $\lambda\geq1$.

Condition ($b^\star$)-($i$) is an immediate consequence of
\[\inf_{x\in\overline{\mathcal{O}_1\times\mathcal{O}_2}}\|L_1x\|=\inf_{x\in D}\|S_1x\|,\]
which follows from the definition of $L$. Lastly, to check condition ($b^\star$)-($ii$), choose $x \in \overline{\mathcal{O}_1 \times \mathcal{O}_2}$ with $x_1 \in \partial_{K_1} \Omega_1$. By the definition of $N$ and condition (A)-($b$)-($ii$) it follows \[x_1-N_1(x_1,x_2)=x_1-T_1(x_1,\rho_2(x_2))\neq \mu\, S_1(x_1,\rho_2(x_2))=\mu\, L_1(x_1,x_2) \quad \text{for} \quad \mu\geq0.\]

These conditions allow us to compute the index of $N$ on $(\mathcal{O}_1\setminus\overline{\Omega}_1) \times (\mathcal{O}_2\setminus\overline{\Omega}_2)$, which coincides with that of $T$ since $T=N$ on this set. Note that ($a^\star$) holding for both $j\in\{1,2\}$ implies that $N$ on $\overline{\mathcal{O}_1\times\mathcal{O}_2}$ satisfies condition ($a$) in Proposition~\ref{prop_ind01}. Hence,
\[i_K(N,\mathcal{O}_1\times\mathcal{O}_2)=1.\]

Restricting the operator $N$ to $\overline{\mathcal{O}_1 \times \Omega_2}$, conditions ($a^\star$) for $j=1$ and ($b^\star$) for $j=2$ place $N$ under the hypotheses of Proposition \ref{prop_ind_sys}. Similarly, exchanging the roles of ($a^\star$) and ($b^\star$) yields the same situation on $\overline{\Omega_1 \times \mathcal{O}_2}$. Hence, 
\[
i_K(N, \mathcal{O}_1 \times \Omega_2) = i_K(N, \Omega_1 \times \mathcal{O}_2) = 0.
\]

From ($b^\star$), it follows that for each $j\in\{1,2\}$
	\[\inf_{x\in \overline{\Omega_1\times\Omega_2}}\|L_jx\|>0; \quad \text{and} \quad x_j-N_jx\neq\mu\, L_j x \ \text{ for } \ x_j\in\partial_{K_j}\Omega_j \ \text{ and } \ \mu\geq 0.\]
Hence, $N$ satisfies situation ($b$) in Proposition~\ref{prop_ind01}. Therefore, 
\[i_K(N,\Omega_1\times\Omega_2)=0.\]

Finally, note that $N$ has no fixed points $x \in \overline{\mathcal{O}_1\times\mathcal{O}_2}$ such that 
	$x_j \in \partial_{K_j}\Omega_j$ or $x_j \in \partial_{K_j}\mathcal{O}_j$ for $j\in\{1,2\}$.
Therefore, we compute the desired index by applying the additivity property of the fixed point index since
\[i_K(N,(\mathcal{O}_1\setminus \overline{\Omega}_1)\times \Omega_2)=i_K(N, \mathcal{O}_1\times \Omega_2)-i_K(N,\Omega_1\times\Omega_2)=0\]
and thus
\[i_K(N,(\mathcal{O}_1\setminus \overline{\Omega}_1)\times(\mathcal{O}_2\setminus \overline{\Omega}_2))=i_K(N, \mathcal{O}_1\times\mathcal{O}_2)-i_K(N,(\mathcal{O}_1\setminus \overline{\Omega}_1)\times \Omega_2)-i_K(N,\Omega_1\times\mathcal{O}_2)=1.\]
Hence, we also have $i_K\left(T,\left(\mathcal{O}_1\setminus \overline{\Omega}_1\right)\times\left(\mathcal{O}_2\setminus \overline{\Omega}_2\right)\right)$=1. The existence property yields the last claim of the theorem. Namely, the existence of a fixed point. 

The remaining cases are analogous, with the main difference being that when one component satisfies (A) and the other (B), the fixed point index changes sign. This follows from the previous computations of the index using the additivity property. Specifically, the relevant set for applying Proposition~\ref{prop_ind01} (a) is $\overline{\mathcal{O}_1 \times \Omega_2}$ if $T_1$ satisfies (A) and $T_2$ satisfies (B), whereas in the reversed situation, it is the set $\overline{\Omega_1 \times \mathcal{O}_2}$. \qed

\begin{remark}
	\label{remark_weak_cond}
	Observe that, we use condition (A)-(b) when we restrict the extended operator to the sets $\mathcal{O}_1\times\Omega_2$ and $\Omega_1\times\mathcal{O}_2$ by applying Proposition \ref{prop_ind_sys}. Also Proposition \ref{prop_ind01} (b) is used when we reduce to $\Omega_1\times\Omega_2$. In this sense condition (A)-(b)-(i), for fixed $j$, could be weakened to
	\[\inf_{x\in (\overline{\mathcal{O}}_i\setminus\Omega_i)\times \partial_{K_j}\Omega_j}\|S_j x\|>0, \quad (i\neq j),\]
	since we can work with the extension to the set $\overline{\mathcal{O}_i\times\Omega_j}$ given by $L_j(x)=S_j(\theta_i(x_i),\rho_j(x_j))$ whenever component $j$ satisfies (A)-(b), and apply Proposition \ref{prop_ind_sys}. After defining $L_j$ for both components in this way, we can consider $L=(L_1,L_2)$ on $\Omega_1\times\Omega_2$ and apply Proposition \ref{prop_ind01} (b) to the operator $T$ in this set.
\end{remark}

\begin{remark}
	\label{remark_versions}
A component under condition (A) in Theorem \ref{th_ppal} is called compressive. Likewise, if instead it fulfills condition (B), it is said to be expansive. This is motivated by the observation that such conditions weaken normed type compression or expansion conditions, respectively. Norm type conditions were introduced by Guo and can be found in \cite{guolak}. The theorem involving them is usually referred to in the literature as the Krasnosel’ski\u{\i}–Guo fixed point theorem. Indeed, let $T:\left(\overline{\mathcal{O}}_1\setminus\Omega_1\right) \times \left(\overline{\mathcal{O}}_2\setminus\Omega_2\right) \rightarrow K$ be a map such that its component $T_j$ $(j\in\{1,2\})$ satisfies in $\left(\overline{\mathcal{O}}_1\setminus\Omega_1\right) \times \left(\overline{\mathcal{O}}_2\setminus\Omega_2\right)$ the following compression Krasnosel'ski\u{\i}–Guo condition
\begin{equation}
	\label{cond_norm}
 \|T_jx\|>\|x_j\| \text{ for } x_j\in \partial_{K_j}\Omega_j \text{ and } \|T_jx\|<\|x_j\| \text{ for } x_j\in\partial_{K_j}\mathcal{O}_j.
\end{equation}
Condition \eqref{cond_norm} ensures that $T$ meets condition (A) of Theorem~\ref{th_ppal} (see \cite{guolak,LFP_JRL} for the details). Thus, Theorem \ref{th_ppal} subsumes the component-wise version of the Krasnosel’ski\u{\i}–Guo theorem \cite[Theorem 3.3]{LFP_JRL}, extending it to operators on regions bounded by strictly star-shaped sets.

 Moreover, this is not the only version covered and generalized by Theorem \ref{th_ppal}. These conditions are also weaker than those known as homotopic, which were employed in \cite{JRL} to prove the Krasnosel’ski\u{\i}–Precup fixed point theorem via the fixed point index. Specifically, $T_j$ is compressive under these homotopic conditions when there exists $h_j\in K_j \setminus\{0\}$ such that $T$ satisfies in $\left(\overline{\mathcal{O}}_1\setminus\Omega_1\right) \times \left(\overline{\mathcal{O}}_2\setminus\Omega_2\right)$: 
\begin{equation}
	\label{cond_homot}
x_j-T_jx\neq \mu\, h_j \text{ for } x_j\in\partial_{K_j}\Omega_j \text{ and }\mu \geq 0, \text{ and } T_jx\neq \lambda\, x_j \text{ for } x_j\in\partial_{K_j}\mathcal{O}_j \text{ and } \lambda\geq 1.
\end{equation}
Condition \eqref{cond_homot} clearly ensures that $T$ fulfills condition (A) of Theorem \ref{th_ppal}. In fact, its first clause reduces to (A)-(b) upon setting $S \equiv h_j$ (i.e., $Sx = h_j$ for every $x$), whereas its second clause is precisely (A)-(a).

Additionally, suppose that $T_j$ is compressive in the classical sense, as first introduced by Krasnosel’ski\u{\i} \cite{Kras} and stated as follows:
\begin{equation}
	\label{cond_clas}
	T_j x \npreceq_j x_j  \text{ for } x_j \in \partial_{K_j} \Omega_j, 
	 \text{ and } 
	T_j x \nsucceq_j x_j  \text{ for } x_j \in \partial_{K_j} \mathcal{O}_j,
\end{equation}
where $\preceq_j$ denotes the partial order induced by the cone $K_j$ on $X_j$, defined by $u\preceq_jv$ if and only if $v-u \in K_j$. It is well-known \cite{guolak} that condition~\eqref{cond_clas} is stronger than condition~\eqref{cond_homot}. Consequently, the latter type of conditions are also included by those in Theorem~\ref{th_ppal}.

It follows analogously that the corresponding expansive conditions imply (B).
\end{remark}

As Theorem \ref{th_ppal} encompasses the different types of conditions presented in the previous remark, it allows for their combination in several ways: by imposing them on different components, and by assigning, for a single component $j$, one condition on the boundary $\partial_{K_j}\Omega_j$ and another on $\partial_{K_j}\mathcal{O}_j$. In this regard, we have, for instance, the following result.

\begin{corollary}
	Suppose that $T=(T_1,T_2):\left(\overline{\mathcal{O}}_1\setminus\Omega_1\right) \times \left(\overline{\mathcal{O}}_2\setminus\Omega_2\right) \rightarrow K$ is a compact map and there exists $h_2\in K_2 \setminus \{0\}$ such that the following conditions are satisfied in $\left(\overline{\mathcal{O}}_1\setminus\Omega_1\right) \times \left(\overline{\mathcal{O}}_2\setminus\Omega_2\right)$:
	\begin{enumerate}[(A)]
		\item $\|T_1x\|>\|x_1\|$ for $x_1\in\partial_{K_1}\Omega_1$ and $T_1x\neq \lambda\, x_1$ for $x_1\in\partial_{K_1}\mathcal{O}_1$ and $\lambda\geq 1$;
		\item $T_2 x \nsucceq_2 x_2$ for $x_2\in\partial_{K_2}\Omega_2$ and $x_2-T_2 x\neq \mu\, h_2$ for $x_2\in\partial_{K_2}\mathcal{O}_2$ and $\mu\geq 0.$
	\end{enumerate}
	Then $i_K(T, (\mathcal{O}_1\setminus\overline{\Omega}_1)\times(\mathcal{O}_2\setminus\overline{\Omega}_2))=-1.$
	
	Moreover, the operator $T$ has at least a fixed point in $(\mathcal{O}_1\setminus\overline{\Omega}_1)\times(\mathcal{O}_2\setminus\overline{\Omega}_2)$.
\end{corollary}

\begin{remark}
Multiplicity results can be obtained by suitably adapting \cite[Theorem 3.4]{LFP_JRL}, leading to a multiplicity result that ensures an odd number of coexistence fixed points located in more general domains.
\end{remark}

\section{Strictly star-shaped sets defined by functionals}

Now, we give sufficient conditions in order to define a strictly star-shaped set over a cone by means of a functional. More precisely, we determine conditions over certain functional $\varphi$ to guarantee that $\overline{\Omega}$ is a strictly star-shaped set in the cone $K$, being $\Omega$ a set of the form
\begin{equation}
	\label{star_funct}
\Omega^{\varphi}_r=\left\{x\in K:\varphi(x)<r \right\},
\end{equation} 
with $r\in\mathbb{R}_+$ and $\varphi:K\rightarrow [0,+\infty)$.

\begin{proposition}
	\label{prop_cond_star}
	Let $K$ be a cone in a normed linear space $X$, $r\in\mathbb{R}_+$, $\varphi:K\rightarrow[0,+\infty)$ a continuous functional and $\Omega^{\varphi}_r$ as given by \eqref{star_funct} .
	
	Then $\overline{\Omega}^{\varphi}_r$ is a strictly star-shaped set provided that the following two conditions hold	
	\begin{enumerate}
		\item[$(C_1)$] there exists $c>0$ such that $c\left\|x\right\|\leq \varphi(x)$ for all $x\in K$;
		\item[$(C_2)$] $\varphi(\lambda\,x)\leq \lambda\,\varphi(x)$ for all $\lambda\in[0,1)$ and all $x\in K$.
	\end{enumerate}
\end{proposition}

\noindent
{\bf Proof.} First, the continuity of the functional $\varphi$ guarantees that $\Omega^\varphi_r:=\varphi^{-1}([0,r))$ is a relatively open set. Second, condition $(C_1)$ implies that $\Omega^\varphi_r$ is a bounded set. Indeed, if $x\in \Omega^\varphi_r$, then $x\in K$ and $\varphi(x)<r$ so, due to $(C_1)$, one has $\left\|x\right\|<r/c$. 

On the other hand, condition $(C_2)$ ensures that $\varphi(0)=0$ and thus $0\in\Omega^\varphi_r$. 
Moreover, if $x\in\partial_K\,\Omega^\varphi_r$, then $\varphi(x)=r$ and thus
\[\varphi(\lambda\,x)\leq\lambda\,\varphi(x)=\lambda\,r<r \quad \text{for all } \lambda\in[0,1), \]
that is, $\lambda\,x\in \Omega^\varphi_r$ for all $\lambda\in[0,1)$. In conclusion, $\overline{\Omega}^\varphi_r$ is a strictly star-shaped set. 
\qed

\begin{remark}
	Note that if a functional $\varphi$ satisfies condition $(C_2)$ above, then it is called a \textit{subhomogeneous functional} following the terminology in \cite{AHL}. 
\end{remark}

\begin{example}
	\label{ex_funct}
	Consider the Banach space of continuous functions $X=\mathcal{C}([0,1])$ endowed with the usual maximum norm $\left\|x\right\|_{\infty}:=\max_{t\in[0,1]}\left|x(t)\right|$. 
	
	It is direct to verify that the following mappings satisfy conditions $(C_1)$ and $(C_2)$ above:
	\begin{enumerate}
		\item[$(i)$] $\varphi:K\rightarrow[0,\infty)$ defined as
		\[\varphi(x)=\alpha\,\min_{t\in[a,b]}\left|x(t)\right|+\beta\,\left\|x\right\|_{\infty},  \]
		with $\alpha,\beta\geq 0$, $\beta\neq 0$, $[a,b]\subset[0,1]$ and being $K$ any cone in $X$.
		\item[$(ii)$] $\varphi:K\rightarrow[0,\infty)$ given by
		\[\varphi(x)=\min_{t\in[a,b]}x(t),  \]
		with $[a,b]\subset[0,1]$ and $K:=\left\{x\in X:\min_{t\in[a,b]}x(t)\geq c\left\|x\right\|_{\infty} \right\}$ for some fixed $c>0$. 
		\item[$(iii)$] $\varphi:K\rightarrow[0,\infty)$ expressed as
		\[\varphi(x)=\max_{t\in[a,b]}x(t),  \]
		with $[a,b]\subset[0,1]$ and $K:=\left\{x\in X:\min_{t\in[a,b]}x(t)\geq c\left\|x\right\|_{\infty} \right\}$ for some fixed $c>0$. 
	\end{enumerate}	
	Observe that the functionals in $(i)$ and $(iii)$ are not concave and so they are not between those considered in \cite[Section 2.2]{JRL}. In addition, the functional given in $(ii)$ does not satisfy condition $(F_5)$ in \cite[Theorem 3.21]{LR} and thus it cannot employed to define an admissible set in the sense of \cite{LR}. 
\end{example}

To illustrate the usefulness of our generalization, we provide an example of an operator whose fixed point can be detected by Theorem~\ref{th_ppal} under norm conditions on star-shaped sets of the form $\overline{\Omega}^\varphi_r$, but not by means of the previous results in annular shells~\cite{LFP_JRL}.

\begin{example}
Let $\psi:\mathbb{R}\rightarrow(-1,1)$ be a continuous function satisfying $\psi(0)=1$, as well as $\psi(t)\geq 0$, when $|t|\leq 1$, and $\psi(t)<0$, when $|t|>1$.

For the cone $P=\{z:=(x,y)\in\mathbb{R}^2: x\geq 0,\, y\geq 0\}$ in $(\mathbb{R}^2, \|\cdot\|_2)$ ( $\|(x,y)\|_2=\sqrt{x^2+y^2}$ for $(x,y)\in\mathbb{R}^2$), consider the operator $T:P\times P\rightarrow P\times P$ given by $T(z_1,z_2)=(\lambda(p(z_1))z_1, \lambda(p(z_2))z_2),$ where $p:\mathbb{R}^2\rightarrow[0,\infty)\times[0,2\pi)$ denotes the polar coordinates transformation $p(x,y)=(r(x,y), \theta(x,y))$ and $\lambda:[0,\infty)\times[0,\pi/2]\rightarrow \mathbb{R}$ is expressed as
\[\lambda(r,\theta)=1+\psi\left(\frac{r-\varphi(\theta)}{\varepsilon}\right),\]
with $\varepsilon\in\mathbb{R}_+$ fixed and $\varphi(\theta)=\frac{1}{\cos\theta+\sin\theta}$.

Let us define $\Omega^\phi_1=\{(x,y)\in P: \phi(x,y)=x+y< 1\}$. As $\phi:P\rightarrow\mathbb{R}$ is under conditions \emph{($C_1$)} and \emph{($C_2$)} in Proposition \ref{prop_cond_star}, $\overline{\Omega}^\phi_1$ is a strictly star-shaped. Observe that $\partial_P\Omega^\phi_1$ can be parameterized as
\[\theta\in[0,\pi/2]\rightarrow \varphi(\theta)(\cos\theta, \sin\theta).\]

For both $j\in\{1,2\}$, we have
\[\|T_jz\|_2>\|z_j\|_2 \text{ for all } z\in P\times P \text{ with } \phi(z_j)=1.\]
Indeed, clearly for $z\in P\times P$ with $\phi(z_j)=1$ we have $T_jz=\lambda(z_j)z_j=(1+\psi(0))z_j=2z_j$. Therefore $\|T_jz\|_2=\|2z_j\|_2=2\|z_j\|_2>\|z_j\|_2$.

Suppose now $\varepsilon<\frac{1-1/\sqrt{2}}{2}$. Then, for all $R\in\mathbb{R}_+$ there exists $\tilde{z}\in P\times P$ with $\|\tilde{z}_j\|_2=R$ and $\|T_j\tilde{z}\|_2\leq \|\tilde{z}_j\|_2$. Clearly, for all $z\in P\times P$ such that $\left\|z_j\right\|=R$ we have $r(z_j)=R$. Furthermore, from the condition over $\varepsilon$ and $\varphi([0,\pi/2])=[1/\sqrt{2},1]$ there exists and angle $\tilde{\theta}\in[0,\pi/2]$ such that $|R-\varphi(\tilde{\theta})|>\varepsilon$. Thus $\lambda(R,\tilde{\theta})<1$. Choosing $\tilde{z}$ such that $p(\tilde{z}_j)=(R,\tilde{\theta})$ it is clear that $\|T_j\tilde{z}\|_2< \|\tilde{z}_j\|_2$. 

Additionally, by taking $\varepsilon=1/10$, we have that 
\[\|T_j z\|_2<\|z_j\|_2 \text{ for all } z\in P\times P \text{ with } \|z_j\|_2=2,\]
as a consequence of having $|2-\varphi(\theta)|>1/10$ for all $\theta\in[0,\pi/2]$.

Note that we have then constructed an example of an operator which, when restricted to the set 
$\left(\overline{P}_{2}\setminus \Omega^\phi_1\right) \times \left(\overline{P}_{2}\setminus \Omega^\phi_1\right)$, 
satisfies the hypothesis of Theorem~\ref{th_ppal} (using the corresponding normed type conditions). 
Nonetheless, the operator never meets the normed type conditions on a set of the form 
$\left(\overline{P}_R \setminus P_r\right) \times \left(\overline{P}_R \setminus P_r\right)$. That is, the fixed point cannot be obtained from the results in \cite{LFP_JRL}.
\end{example}

Note that in the previous example, \(\phi\) coincides with the norm \(\|\cdot\|_1\) in \(\mathbb{R}^2\) over the cone \(P\), since 
\(\|(x,y)\|_1 = |x| + |y|\) for all \((x,y) \in \mathbb{R}^2\). In particular, norms are examples of functionals satisfying the conditions in Proposition~\ref{prop_cond_star}. 

Clearly, Theorem~\ref{th_ppal} allows us to establish a result in two different norms in the following sense.

\begin{remark}
 For a fixed $j \in \{1,2\}$ let $|\cdot|_j$ be another norm of $(X_j,\|\cdot\|_j)$. Assume that there exists $c\in\mathbb{R}_+$ such that $ \|x\|_j\leq c|x|_j$ for all $x\in K_j$. Then for  $r_j, R_j \in \mathbb{R}_+$ with $c\, r_j<R_j$ it is clear that $\overline{\Omega}^{|\cdot|_j}_{r_j}\subset\Omega^{\|\cdot\|_j}_{R_j}$.
 
 Observe that the compressive normed type condition in Theorem \ref{th_ppal} for component $j$, can be written as
 \[\|T_j x\|>\|x_j\| \text{ for } |x_j|=r_j \text{ and } \|T_j x\|<\|x_j\| \text{ for } \|x_j\|=R_j.\]
In fact, we can go further and impose the condition on the operator in terms of both norms, that is,
\begin{equation}
	\label{cond_two_norm}
	|T_j x| > |x_j| \text{ for } |x_j| = r_j 
	\quad \text{and} \quad 
	\|T_j x\| < \|x_j\| \text{ for } \|x_j\| = R_j.
\end{equation}
Under this compressive condition \eqref{cond_two_norm} for both components, the operator \(T\) satisfies the assumptions of Theorem~\ref{th_ppal}(A) for each \(j \in \{1,2\}\), in the set $\left(\overline{\Omega}^{\|\cdot\|_1}_{R_1}\setminus\Omega^{|\cdot|_1}_{r_1}\right) \times \left(\overline{\Omega}^{\|\cdot\|_2}_{R_2}\setminus\Omega^{|\cdot|_2}_{r_2}\right)$. 

The arguments are similar to those discussed in Remark~\ref{remark_versions} for the classical normed version. It is also straightforward to formulate the expansion condition for component \(j\) and verify that \(T\) satisfies (B). 

Hence, we obtain an extension to the setting of systems of the compression–expansion fixed point theorem in two norms stated by O'Regan and Precup in~\cite{PrecupMAA}. Moreover, this result not only guarantees the existence of a fixed point (as does \cite{PrecupFPT}) but also provides the value of the fixed point index of the operator.
\end{remark}

When strictly star-shaped sets are defined by functionals, conditions on the operators can be formulated in terms of them as well as shown in the following result inspired by \cite[Theorem 8]{fel}.

\begin{theorem}
	\label{the_funct}
	For $j\in\{1,2\}$, let $r_j,R_j\in\mathbb{R}_+$ be such that $r_j<R_j$ and $\varphi_j, \psi_j:K_j\rightarrow [0,+\infty)$ be continuous functionals under conditions $(C_1)$ and $(C_2)$. 
	Assume that $T:\left(\overline{\Omega}^{\psi_1}_{R_1}\setminus\Omega^{\varphi_1}_{r_1}\right)\times\left(\overline{\Omega}^{\psi_2}_{R_2}\setminus\Omega^{\varphi_2}_{r_2}\right)\rightarrow K$ is a compact map such that, for $j\in\{1,2\}$, one of the following conditions holds
	\begin{enumerate}[(A)]
		\item $\varphi_j(T_jx)>r_j$ if $\varphi_j(x_j)=r_j$ and $\psi_j(T_jx)<R_j$ if $\psi_j(x_j)=R_j$;
		\item $\varphi_j(T_jx)<r_j$ if $\varphi_j(x_j)=r_j$ and $\psi_j(T_jx)>R_j$ if $\psi_j(x_j)=R_j$.
	\end{enumerate}
	Then $i_K\left(T,\left({\Omega}^{\psi_1}_{R_1}\setminus\overline{\Omega}^{\varphi_1}_{r_1}\right)\times\left({\Omega}^{\psi_2}_{R_2}\setminus\overline{\Omega}^{\varphi_2}_{r_2}\right)\right)=(-1)^k$, where $k\in\{0,1,2\}$ is a counter of the number of indexes $j$ for which condition (B) is satisfied.
	
	In particular, the operator $T$ has at least one fixed point in $\left({\Omega}^{\psi_1}_{R_1}\setminus\overline{\Omega}^{\varphi_1}_{r_1}\right)\times\left({\Omega}^{\psi_2}_{R_2}\setminus\overline{\Omega}^{\varphi_2}_{r_2}\right)$.
\end{theorem}

\noindent
{\bf Proof.} Suppose, for instance, that for both $j\in\{1,2\}$ condition (A) holds. Then $T$ is under the same condition in Theorem \ref{th_ppal}. In fact, suppose that we do not have for $x\in\left(\overline{\Omega}^{\psi_1}_{R_1}\setminus\Omega^{\varphi_1}_{r_1}\right)\times\left(\overline{\Omega}^{\psi_2}_{R_2}\setminus\Omega^{\varphi_2}_{r_2}\right)$ that
\[T_jx\neq\lambda\, x_j \text{ for } \psi_j(x_j)=R_j \text{ and } \lambda\geq 1.\]
Then there exists \( x \in \left(\overline{\Omega}^{\psi_1}_{R_1}\setminus\Omega^{\varphi_1}_{r_1}\right)\times\left(\overline{\Omega}^{\psi_2}_{R_2}\setminus\Omega^{\varphi_2}_{r_2}\right) \) with \( \psi_j(x_j) = R_j \) and \( \lambda \ge 1 \) such that \( T_j x = \lambda\, x_j \). 
By condition \( (C_2) \) for \( \psi_j \), we have
\[
R_j = \psi_j(x_j) = \psi_j\!\left(\frac{1}{\lambda} T_j x\right) 
\le \frac{1}{\lambda}\psi_j(T_j x).
\]
Hence \( \psi_j(T_j x) \ge \lambda R_j \ge R_j \), 
which contradicts the hypothesis.

Similarly, we have that the operator satisfies for $x\in\left(\overline{\Omega}^{\psi_1}_{R_1}\setminus\Omega^{\varphi_1}_{r_1}\right)\times\left(\overline{\Omega}^{\psi_2}_{R_2}\setminus\Omega^{\varphi_2}_{r_2}\right)$ that
\[x_j - T_j x \neq \mu\, T_j x 
\quad \text{for } \varphi_j(x_j) = r_j \text{ and } \mu \ge 0.\]
Otherwise, there would exist \( x \in \left(\overline{\Omega}^{\psi_1}_{R_1}\setminus\Omega^{\varphi_1}_{r_1}\right)\times\left(\overline{\Omega}^{\psi_2}_{R_2}\setminus\Omega^{\varphi_2}_{r_2}\right) \) with 
\( \varphi_j(x_j) = r_j \) and \( \mu \ge 0 \) such that 
\( x_j = (1+\mu)\, T_j x \). 
Again, by condition \( (C_2) \) for \( \varphi_j \), we obtain
\[ r_j < \varphi_j(T_j x) 
= \varphi_j\!\left( \frac{1}{1+\mu} x_j \right) 
\le \frac{1}{1+\mu} \varphi_j(x_j) 
= \frac{r_j}{1+\mu},\]
a contradiction.

Finally, it is clear that 
\[ \inf_{x_i\in \overline{\Omega}^{\psi_i}_{R_i}\setminus\Omega^{\varphi_i}_{r_i},\, x_j\in \partial_{K_j}\Omega^{\varphi_j}_{r_j}} \|T_j x\| > 0 
\quad (j \ne i). \]
Indeed, if this were not the case, there would exist $x=(x_1,x_2)$ such that
\( x_i\in \overline{\Omega}^{\psi_i}_{R_i}\setminus\Omega^{\varphi_i}_{r_i},\, x_j\in \partial_{K_j}\Omega^{\varphi_j}_{r_j} \) 
with \( \|T_j x\| = 0 \). 
Then necessarily \( T_j x = 0 \), and hence \( \varphi_j(T_j x) = 0 \) by condition \( (C_2) \), 
a contradiction. The conclusion holds in light of the Remark \ref{remark_weak_cond}. 
%{\color{red} To check (A)-$(b)$ in Theorem~\ref{th_ppal} it suffices to define $S_j$ as the extension of $T_j$ provided by the retraction from $\overline{\Omega}^{\varphi_j}_{r_j}$ into $\partial_{K_j}\Omega^{\varphi_j}_{r_j}$.
%}
\qed

\section{Applications}

\subsection{Hammerstein systems}
Consider a Hammerstein-type system of equations of the form
\begin{equation}
	\label{hammer_sys}
	\begin{cases}
		x_1(t) = \displaystyle \int_0^1 G_1(t,s) f_1(s, x_1(s), x_2(s))\, ds, & t \in [0,1],\\[1.2ex]
		x_2(t) = \displaystyle \int_0^1 G_2(t,s) f_2(s, x_1(s), x_2(s))\, ds, & t \in [0,1],
	\end{cases}
\end{equation}
where for each $j \in \{1,2\}$ the following conditions are satisfied:
\begin{enumerate}
	\item[$(\mathcal{H}_1)$] The kernel $G_j: I^2 \to \mathbb{R}^+$ is continuous ($I:=[0,1]$ and $\mathbb{R}^+:=[0,\infty)$).
	\item[$(\mathcal{H}_2)$] There exists a close interval $J \subset I$, a function $\Phi_j:I \to \mathbb{R}^+$ such that 
	\[
	\Phi_j \in L^1(0,1), \qquad \int_J \Phi_j(s)\, ds > 0,
	\]
	and a constant $c_j \in (0,1)$ satisfying
	\[
	G_j(t,s) \le \Phi_j(s), \quad \text{for all } t,s \in I,
	\]
	\[
	c_j \Phi_j(s) \le G_j(t,s), \quad \text{for all } t \in J,\ s \in I.
	\]
	
	\item[$(\mathcal{H}_3)$] The function $f_j :I \times \mathbb{R}^+ \times \mathbb{R}^+ \to \mathbb{R}^+$ is continuous.
\end{enumerate}

For $j\in\{1,2\}$, let us introduce the cone $P_j=\{x\in\mathcal{C}(I): x(t)\geq 0 \text{ for } t\in I\},$ of the Banach space of continuous functions $\mathcal{C}(I)$ with the maximum norm $\|\cdot\|_\infty$.

To establish the existence of a solution to the integral system \eqref{hammer_sys}, we apply our fixed point results to the operator $T : K \to K, \, T = (T_1, T_2),$ defined by
\begin{equation}
	\label{oper}
	T_j(x_1, x_2)(t) := \int_0^1 G_j(t,s) f_j(s, x_1(s), x_2(s))\, ds, 
	\quad t \in I, \quad  (j = 1,2),
\end{equation}
with $K:= K_1 \times K_2$, where for $j\in\{1,2\}$, $K_j$ is the cone given by
\[
K_j = \left\{ x \in P_j: x(t) \ge c_j \|x\|_\infty, \text{ for } t\in J \right\}.
\]

By standard arguments (see \cite{Figue_tojo,web}), under assumptions $(\mathcal{H}_1)-(\mathcal{H}_3)$ one can show that the operator $T$ leaves the cone $K$ invariant and is compact, that is, $T$ is continuous and maps bounded sets into relatively compact ones.

Henceforth, we use the following notation:
\[d_j:=\left(\max_{t\in[0,1]}\int_0^1 G_j(t,s)ds\right)^{-1}, \quad D_j:=\left(\min_{t \in J}\int_J G_j(t,s)ds\right)^{-1}, \quad (j=1,2),\]
and 
\[S_j:=\max_{t\in[0,1]}\int_J G_j(t,s)ds, \quad S^c_j:=\max_{t\in[0,1]}\int_{J^c} G_j(t,s)ds,\quad (j=1,2),\]
where $J^c=I\setminus J$.

Additionally, for each $j\in\{1,2\}$, let $\varphi_j:K_j\rightarrow\mathbb{R}^+$ be the functional defined by
\begin{equation}\label{eq_funct}
\varphi_j(x)=\min_{t \in J} x(t),
\end{equation}
and denote
\[m_j^{r,R}(s):=\min\left\{f_j(s,x_1,x_2): r_j\leq x_j\leq \frac{r_j}{c_j}, \, r_i\leq x_i\leq R_i\, (i\neq j)\right\},\]
\[M^{R}_j(s):=\max\{f_j(s,x_1,x_2): 0\leq x_j\leq  R_j,\, 0\leq x_i\leq R_i\,(i\neq j)\},\]
\[M^{r,R}_j(s):=\max\{f_j(s,x_1,x_2): c_jR_j\leq x_j\leq R_j, \, r_i\leq x_i\leq R_i\, (i\neq j)\},\]
for each $s\in I$.

Note that for $0<r_j<c_jR_j$ we have that $\overline{\Omega}_{r_j}^{\varphi_j}\subset (K_j)_{R_j}$, being $\overline{\Omega}_{r_j}^{\varphi_j}$ the strictly star-shaped set over $K_j$ defined as in \eqref{star_funct} by means of the functional $\varphi_j$ in \eqref{eq_funct} and $(K_j)_{R_j}:=\{x_j\in K_j: \|x_j\|<R_j\}$.

At this point, we are in position to establish an existence result for the system (\ref{hammer_sys}) as a consequence of our Theorem \ref{th_ppal}.

\begin{theorem}
	\label{th_loc1}
Let hypotheses $(\mathcal{H}_1)-(\mathcal{H}_3)$ be satisfied. Suppose that there exists $r_j,R_j,A_j,B_j\in\mathbb{R}_+$, $j=1,2$, such that $r_j<c_jR_j$, $B_j\leq d_j\leq A_j$ and $B_j S_j+A_j S^c_j<1$. Assume further that, for each $j\in\{1,2\}$, the following conditions hold:
\begin{enumerate}[$(A)$]
	\item $0\leq M^{R}_j(s)\leq A_j R_j$ for $s\in J^c$ and $0\leq M_j^{r,R}(s)\leq B_jR_j$ for $s\in J$;
	\item $m_j^{r,R}(s)> D_jr_j$ for $s\in J$.
\end{enumerate}

Then the system \eqref{hammer_sys} has at least one positive solution $(x_1,x_2)\in K$ such that $r_j<\varphi_j(x_j)$ and $\|x_j\|_{\infty}<R_j, \, (j=1,2)$.
\end{theorem}

\noindent
{\bf Proof.} Take $x\in\left((\overline{K}_1)_{R_1}\setminus\Omega^{\varphi_1}_{r_1}\right)\times\left((\overline{K}_2)_{R_2}\setminus\Omega^{\varphi_2}_{r_2}\right)$ with $\|x_j\|_\infty=R_j$. Hence $0\leq x_j(t)\leq R_j$ for $ t\in[0,1]$ and $c_jR_j\leq x_j(t)\leq R_j $ for $ t\in J$ whereas for $i\neq j$ we have $0\leq x_i(t)\leq R_i, \, t\in[0,1]$  and $r_i\leq x_i(t)\leq R_i,\, t\in J $. For $t\in[0,1]$
\[T_jx=\int_J G_j(t,s)f_j(s,x(s))ds+\int_{J^c}G_j(t,s)f_j(s,x(s))ds\leq \int_J G_j(t,s)B_jR_jds+\int_{J^c}G_j(t,s)A_jR_jds.\]
Taking the supremum for $t\in[0,1]$ gives $\|T_jx\|_\infty\leq (B_jS_j+A_jS_j^c)R_j<R_j$.

We have then $\|T_j x\|_\infty<\|x_j\|_\infty$ for $x\in\left((\overline{K}_1)_{R_1}\setminus\Omega^{\varphi_1}_{r_1}\right)\times\left((\overline{K}_2)_{R_2}\setminus\Omega^{\varphi_2}_{r_2}\right)$ with $\|x_j\|_\infty=R_j$, i.e., condition (A)-(a) in Theorem \ref{th_ppal} holds.

Let $h_j(t)=1$ for $t\in[0,1]$. Then $h_j\in K_j$ and $\|h_j\|_\infty=1$. We show that
\[x_j\neq T_j x+\lambda h_j \text{ for } x\in\left((\overline{K}_1)_{R_1}\setminus\Omega^{\varphi_1}_{r_1}\right)\times\left((\overline{K}_2)_{R_2}\setminus\Omega^{\varphi_2}_{r_2}\right) \text{ with } \varphi_j(x_j)=r_j \text{ and } \lambda\geq 0,\]
which corresponds to condition (A)-(b) in Theorem \ref{th_ppal}, and thus we complete the proof. In fact, if not, there exist $x\in\left((\overline{K}_1)_{R_1}\setminus\Omega^{\varphi_1}_{r_1}\right)\times\left((\overline{K}_2)_{R_2}\setminus\Omega^{\varphi_2}_{r_2}\right)$ with $\varphi_j(x_j)=r_j$ and $\lambda\geq 0$ such that $x_j=T_jx+\lambda h_j$. Observe that for $x\in\left((\overline{K}_1)_{R_1}\setminus\Omega^{\varphi_1}_{r_1}\right)\times\left((\overline{K}_2)_{R_2}\setminus\Omega^{\varphi_2}_{r_2}\right)$ with $\varphi_j(x_j)=r_j$ we have $r_j\leq x_j(t)\leq r_j/c_j$ and $r_i\leq x_i(t)\leq R_i$ for $t\in J$ ($i\neq j$). Then, for $t\in J$,
\[x_j(t)=\int_{0}^{1}G_j(t,s)f_j(s,x(s))ds+\lambda > D_jr_j\int_{J} G_j(t,s)ds+\lambda\geq r_j+\lambda\geq r_j.\]
Since $\varphi_j(x_j)=r_j$, this is a contradiction. \qed

The following theorem localizes the solution in annular shells and can be derived from those results in previous works \cite{PrecupFPT,JRL}. Once proved, it will be compared with Theorem~\ref{th_loc1}.

\begin{theorem}
	\label{th_loc2}
	Let hypotheses $(\mathcal{H}_1)-(\mathcal{H}_3)$ be satisfied. Suppose that there exists $\tilde{r}_j,\tilde{R}_j,A_j,B_j\in\mathbb{R}_+$, $j=1,2$, such that $\tilde{r}_j<\tilde{R}_j$, $B_j\leq d_j\leq A_j$ and $B_j S_j+A_j S^c_j<1$. Assume further that, for each $j\in\{1,2\}$, the following conditions hold:
	\begin{enumerate}[$(A)$]
		\item $0\leq M_j^{\tilde{R}}(s)\leq A_j \tilde{R}_j$ for $s\in J^c$ and $0\leq M_j^{c\tilde{r},\tilde{R}}(s)\leq B_j \tilde{R}_j$ for $s\in J$;
		\item $m_j^{c\tilde{r},\tilde{R}}(s)> D_j \tilde{r}_j$ for $s\in J$. %where
		%\[\tilde{m}_j^{\tilde{r},\tilde{R}}(s)=\min\{f_j(s,x_1,x_2): c_j\tilde{r}_j\leq x_j\leq \tilde{r}_j, \, c_i\tilde{r}_i\leq x_i\leq \tilde{R}_i\, (i\neq j)\}.\]
	\end{enumerate}
	
Then the system \eqref{hammer_sys} has at least one positive solution $(x_1,x_2)\in K$ such that $\tilde{r}_j<\|x_j\|_\infty<\tilde{R}_j\, (j=1,2)$. 
\end{theorem}

\noindent
{\bf Proof.} Following the previous proof, we already know that condition (A)-(a) in Theorem \ref{th_ppal} holds for operator $T$ given by (\ref{oper}) in $\overline{K}_{\tilde{r},\tilde{R}}=\left((\overline{K}_1)_{\tilde{R}_1}\setminus (K_1)_{\tilde{r}_1}\right)\times\left((\overline{K}_2)_{\tilde{R}_2}\setminus (K_2)_{\tilde{r}_2}\right)$. It only remains to see that
\[x_j\neq T_j x+\lambda h_j \text{ for } x\in \overline{K}_{\tilde{r},\tilde{R}} \text{ with } \|x_j\|_\infty=\tilde{r}_j \text{ and } \lambda\geq 0,\]
where $h_j$ is the same function than in the previous proof. Suppose this is not true, then there exist $x\in\overline{K}_{\tilde{r},\tilde{R}}$ with $\|x_j\|_\infty=\tilde{r}_j$ and $\lambda\geq 0$ such that $x_j=T_jx+\lambda h_j$. Observe that for $x\in\overline{K}_{\tilde{r},\tilde{R}}$ with $\|x_j\|_\infty=\tilde{r}_j$ we have $c_j\tilde{r}_j\leq x_j(t)\leq \tilde{r}_j$  and $c_i\tilde{r}_i\leq x_i(t)\leq \tilde{R}_i\, (i\neq j)$ for $t\in J$. Then, for $t\in J$
\[x_j(t)=\int_0^1 G_j(t,s)f_j(s,x(s))ds+\lambda > D_j \tilde{r}_j\int_{J} G_j(t,s)ds+\lambda\geq \tilde{r}_j+\lambda\geq \tilde{r}_j.\]
Since $\|x_j\|_\infty=\tilde{r}_j$, it gives a contradiction. \qed

\begin{remark}
Observe that Theorems \ref{th_loc1} and \ref{th_loc2} are comparable when
\begin{equation}
	\label{cond_comp}
r_j = c_j \tilde{r}_j \text{ and } R_j = \tilde{R}_j, \quad j = 1,2,
\end{equation}
since in this case we have $m_j^{r,R}(s)=m_j^{c\tilde{r},\tilde{R}}(s)$ for $s\in[0,1]$ and $M_j^{r,R}(s)=M_j^{c\tilde{r},\tilde{R}}(s)$ for $s\in J$. In this situation, the hypotheses of Theorem~\ref{th_loc1} are weaker than those of Theorem~\ref{th_loc2}.  
Let us reflect on how this is related to the localization they provide.

Notice that, for $r_j \in \mathbb{R}_+$ ($j=1,2$), we have
\[
(K_j)_{r_j} \subset \Omega^{\varphi_j}_{r_j} \subset (K_j)_{\frac{r_j}{c_j}}, \quad j=1,2,
\]
since if $x_j \in (K_j)_{r_j}$, then $\|x_j\|_\infty = \max_{t \in I} x_j(t) < r_j$ and thus clearly $\min_{t \in J} x_j(t) < r_j$. Moreover, if $x_j\in \Omega^{\varphi_j}_{r_j}$, then $c_j\,\|x_j\|_\infty\leq \min_{t \in J} x_j(t) < r_j$ and so $\|x_j\|_\infty<r_j/c_j$.

\begin{figure}[h]
	\centering
	\includegraphics[scale=1]{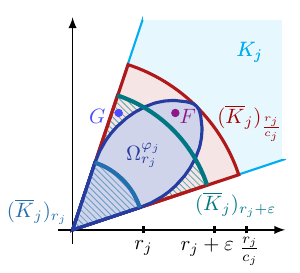}
	\vspace{-0.5cm}
	\caption{An illustrative idea of the relative position of the sets $(\overline{K}_j)_{r_j}$, $\overline{\Omega}_{r_j}^{\varphi_j}$ and $(\overline{K}_j)_{\frac{r_j}{c_j}}$.}
	\label{fig2}
\end{figure} 

Hence, for fixed $r_j, R_j \in \mathbb{R}_+$ ($j = 1,2$) in the setting of Theorem~\ref{th_loc1}, 
if we choose $\tilde{r}_j, \tilde{R}_j\in\mathbb{R}_+\,$($j=1,2$) in Theorem~\ref{th_loc2} according to~\eqref{cond_comp}, 
then Theorem~\ref{th_loc2} locates the fixed point in a subset of the localization set given by Theorem~\ref{th_loc1}. 
Therefore, since the localization set in Theorem~\ref{th_loc1} is larger, 
it is natural that its assumptions are weaker than those in Theorem~\ref{th_loc2}. 

On the other hand, if condition \eqref{cond_comp} does not hold, the previous existence results are not comparable, since the hypotheses imposed on $f_j$ cannot be compared. In addition, fixing $\varepsilon \in \left(0, \frac{r_j}{c_j}-r_j\right)$, it is possible to find elements in $(K_j)_{r_j+\varepsilon}$ that are not in $\Omega_{r_j}^{\varphi_j}$, as well as elements in $\Omega_{r_j}^{\varphi_j}$ that are not in $(K_j)_{r_j+\epsilon}$. Therefore,
\[
\Omega^{\varphi_j}_{r_j} \not\subset (K_j)_{r_j+\epsilon} \quad \text{and} \quad (K_j)_{r_j+\epsilon} \not\subset \Omega^{\varphi_j}_{r_j}.
\]

\begin{figure}[h]
	\centering
	\subfigure[$F\in \Omega_{r_j}^{\varphi_j}$ and $F\notin (K_j)_{r_j+\varepsilon}$.]{\includegraphics[scale=1.1]{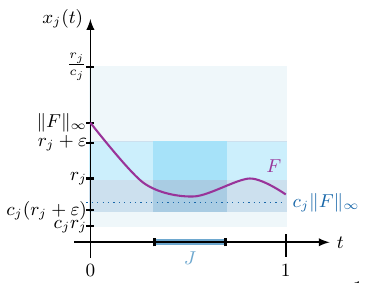}}
	\hspace{1cm}
	\subfigure[$G\in (K_j)_{r_j+\varepsilon}$ and $G\notin\Omega_{r_j}^{\varphi_j}$]{\includegraphics[scale=1.1]{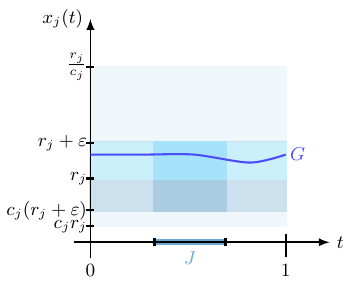}
	}
	\caption{Graphical representation of mutual non-inclusion between $\Omega_{r_j}^{\varphi_j}$ and $(K_j)_{r_j+\varepsilon}$ (see Figure~\ref{fig2}).}
\end{figure}

Consequently, if $r_j, \tilde{r}_j, R_j$ and $\tilde{R}_j$ $(j=1,2)$ are not chosen as stated in \eqref{cond_comp}, then there is no containment relation between the localization sets provided by Theorems~\ref{th_loc1} and~\ref{th_loc2}. 
This is consistent with the fact that the results are not comparable. 
\end{remark}

Under the assumptions of Theorems \ref{th_loc1} and \ref{th_loc2}, both components of the operator $T$ are compressive. In our next result the operator satisfies a mixed behavior (expansive in one component and compressive in the other). Moreover, it provides a new localization of the solution (the proof is omitted, as it follows similar reasoning to those in previous results). Note that employing different functionals, in turn, yields distinct existence results.

To this end, we introduce the following constant
\[s_j:=\max_{t\in J}\int_J G_j(t,s)ds, \quad s^c_j:=\max_{t\in J}\int_{J^c} G_j(t,s)ds, \quad (j=1,2).\]
Consider also, for $j\in\{1,2\}$, the functional $\psi:K_j\rightarrow\mathbb{R}^+$ defined by
\[\psi_j(x)=\max_{t \in J} x(t).\]

Observe that for $r_j,R_j\in\mathbb{R}_+\, (j=1,2)$ such that $r_j<c_jR_j$ we have $\overline{\Omega}^{\psi_j}_{r_j}\subset(K_j)_{R_j}$.

\begin{theorem}
	Let assumptions $(\mathcal{H}_1)-(\mathcal{H}_3)$ hold. Suppose that there exists $r_j,R_j,A_j,B_j\in\mathbb{R}_+$, $j=1,2$, such that $r_1<c_1R_1$, $r_2<c_2R_2$, $B_1\leq d_1\leq A_1$, $A_2\leq d_2\leq B_2$, $B_1S_1+A_1 S^c_1<1$ and $A_2 s_2+B_2 s^c_2<1$. Assume further that, for each $j\in\{1,2\}$, the following conditions hold:
	\begin{enumerate}[$(A)$]
		\item 
		\begin{enumerate}[$(a)$]
			\item $0\leq M_1^{R}(s)\leq A_1R_1$ for $s\in J^c$ and $0\leq M_1^{cr,R}(s)\leq B_1R_1$ for $s\in J$;
			\item $\tilde{m}_1^{r,R}(s)>D_1r_1$ for $s\in J$, where
			\[\tilde{m}_1^{r,R}(s)=\min\left\{f_1(s,x_1,x_2): r_1\leq x_1\leq \frac{r_1}{c_1}, \, c_2r_2\leq x_2\leq R_2\right\}.\]
		\end{enumerate}
		\item 
		\begin{enumerate}[$(a)$]
			\item $0\leq M_2^{r}(s)\leq A_2r_2$ for $s\in J$ and $0\leq \hat{M}_2^{r,R}(s)\leq B_2r_2$ for $s\in J^c$, where
			\[\hat{M}_2^{r,R}(s)=\max\{f_2(s,x_1,x_2): r_1\leq x_1\leq R_1,\, c_2r_2\leq x_2\leq R_2\};\]
			\item $\hat{m}_2^{r,R}(s)> D_2R_2$ for $s\in J$, where
			\[\hat{m}_2^{r,R}(s)=\min\left\{f_2(s,x_1,x_2): r_1\leq x_1\leq R_1, \, c_2R_2\leq x_2\leq R_2\right\}.\]
		\end{enumerate}
	\end{enumerate}
	
	Then the system \eqref{hammer_sys} has at least one positive solution $(x_1,x_2)\in K$ such that $\varphi_1(x_1)>r_1, \, \|x_1\|_\infty<R_1$ and $\psi_2(x_2)>r_2, \,  \|x_2\|_\infty<R_2$.
\end{theorem}

Finally, we use a functional of the form $(i)$ in Example~\ref{ex_funct} to establish a new existence result with a different localization. Let us then, define for $j\in\{1,2\}$, $\phi_j:K_j\rightarrow \mathbb{R}^+$ given by
\[\phi_j(x)=\frac{1}{2}\left(\min_{t \in J}x_j(t)+\|x_j\|_\infty\right).\]

Note that for $r_j,R_j\in\mathbb{R}^2_+\, (j=1,2)$, if $\frac{2}{c_j+1}r_j<R_j$ then $\overline{\Omega}^{\phi_j}_{r_j}\subset (K_j)_{R_j}$.

\begin{theorem}
	\label{th_loc4}
 	Let assumptions $(\mathcal{H}_1)-(\mathcal{H}_3)$ hold. Suppose that there exists $r_j,R_j,A_j,B_j\in\mathbb{R}_+$, $j=1,2$, such that $\frac{2}{c_j+1}r_j<R_j, \, B_j\leq d_j\leq A_j$ and $B_jS_j+A_jS^c_j<1$. Assume further that, for each $j\in\{1,2\}$, the following conditions hold:
 	\begin{enumerate}[$(A)$]
 		\item $0\leq M_j^{R}(s)\leq A_jR_j$ for $s\in J^c$ and $0\leq M_j^{\frac{2c}{c+1}r,R}(s)\leq B_jR_j$ for $s\in J$;
 		\item $\breve{m}_j^{r,R}(s)>D_jr_j$ for $s\in J$, where
 		\[\breve{m}_j^{r,R}(s)=\min\left\{f_j(s,x_1,x_2): \frac{2c_j}{c_j+1}r_j\leq x_j\leq \frac{2}{c_j+1}r_j, \, \frac{2c_i}{c_i+1}r_i\leq x_i\leq R_i\, (i\neq j)\right\}.\]
 	\end{enumerate}
 	
 	Then the system \emph{(\ref{hammer_sys})} has at least one positive solution $(x_1,x_2)\in K$ such that $\phi_j(x_j)>r_j$ and $\|x_j\|_\infty<R_j, \,(j=1,2).$
\end{theorem}

\begin{remark}
	\label{re5.2}
Note that, throughout this section, we assume that the nonlinearities $f_j$ are nonnegative on $[0,1]\times\mathbb{R}^+\times\mathbb{R}^+$. This condition is imposed to guarantee that the operator $T$ maps the cone $K$ into itself. Actually, it suffices to require that
\[
f_j(t,x_1,x_2)\geq 0, \quad \text{for all } t\in[0,1] \text{ and } (x_1,x_2)\in[0,R_1]\times[0,R_2],
\]
to guarantee that the operator maps into $K$ the corresponding sets on which Theorem~\ref{th_ppal} is applied in each case.
\end{remark}

\begin{example}
Consider the following system of second-order equations, where one component is subject to Dirichlet boundary conditions and the other to Dirichlet–Neumann mixed boundary conditions.
\begin{equation}
	\label{sys_diff}
	\begin{cases}
		x''(t) + f_1(t,x,y) = 0,\quad t\in[0,1],\\
		y''(t) + f_2(t,x,y) = 0,\quad t\in[0,1],\\
		x(0) = x(1) = 0 = y(0) = y'(1),
	\end{cases}
\end{equation}
with $f_1(t,x,y) = \bigl(2+\frac{y}{10}\bigr)e^{4-\sqrt{x+1}}+\frac{t}{2}$, $f_2(t,x,y)=\sqrt[100]{y}\left(10g(y)+\frac{10+t^2}{1+e^{-x}}\right)$
and
\[
g(y) =
\begin{cases}
	e^{-10y}, & \text{if } y\in[0,10],\\
	e^{-10y}+\sin(y-10), & \text{if } y>10.
\end{cases}
\]
We can associate to \eqref{sys_diff} a system of Hammerstein type equations of the form \eqref{hammer_sys}, where the kernels are given by the corresponding Green’s functions
\[
G_1(t,s)=
\begin{cases}
	s(1-t), & \text{if } s\le t,\\
	t(1-s), & \text{if } s>t,
\end{cases}
\qquad
G_2(t,s)=
\begin{cases}
	s, & \text{if } s\le t,\\
	t, & \text{if } s>t.
\end{cases}
\]
It is well-known (see, for instance, \cite{infante}) that condition $(\mathcal{H}_2)$ holds if we take $\Phi_1(s)=s(1-s)$, $\Phi_2(s)=s$, $J=[1/4,3/4]$ and $c_j=1/4$ $(j=1,2)$. This choice leads to $d_1=D_2=8$, $d_2=2$, $D_1=16$, $S_1=3/32$, $S^c_1=1/32$, $S_2=22/32$ and $S^c_2=9/32$. 

Observe that $y(t) = 0$ for $t \in [0,1]$ satisfies the second equation together with its associated boundary conditions. Furthermore, the following second-order boundary value problem  
	\begin{equation}
	\label{sys_diff2}
		\begin{cases}
		x''(t) + \tilde{f}_1(t,x) = 0,\quad t\in[0,1],\\
		x(0) = x(1) = 0 ,
	\end{cases}
	\end{equation}
	where $\tilde{f}_1(t,x) = f_1(t,x,0)$ admits a solution $\tilde{x} \in K_1$ satisfying $1 \le \|\tilde{x}\|_\infty \le 6$ (see \cite[Theorem 1.0.7]{infante}, suitably adapted to the non-autonomous setting). Consequently, $(\tilde{x},0)$ constitutes a solution for our system \eqref{sys_diff}.  
	
	We now proceed to identify a second, distinct solution by appealing to our previous result. The localization ensured for this second solution excludes the case in which any of its components is zero.

In connection with Remark \ref{re5.2}, observe that $f_1$ is nonnegative on $[0,1]\times\mathbb{R}^+\times\mathbb{R}^+$, and $f_2$ on $[0,1]\times[0,10]^2$.

Take $r_1=1$, $r_2=1/4$, $R_1=10$ and $R_2=20$. Choose $A_1=9$ and $B_1=4$. Observe that $B_1\leq d_1\leq A_1$ and $A_1S^c_1+B_1S_1=21/32<1$. Then $0\leq M_1^{R}(s)\leq 90=A_1R_1, $ for $s\in[0,1]$, $0\leq M_1^{\frac{2c}{c+1}r,R}(s)\leq 40=B_1R_1,$ for $s\in J$ and $\breve{m}_1^{r,R}(s)>16=D_1r_1,$ for $s\in J$.
The first component of the operator $T$ satisfies conditions (A) and (B) of Theorem~\ref{th_loc4}. We shall show that these conditions also hold for the second component by taking $A_2 = 2$ and $B_2 = 3/5$. Clearly, $B_2\leq d_2\leq A_2$ and $A_2S^c_2+B_2S_2<1$. In addition
\[0\leq M_2^{R}(s)\leq 40=A_2R_2, \text{ for } s\in[0,1],\]
\[0\leq M_2^{\frac{2c}{c+1}r,R}(s)\leq 12=B_2R_2, \text{ and }  \breve{m}_2^{r,R}(s)>2=D_2r_2,  \text{ for } s\in J.\]
By Theorem \ref{th_loc4}, we know that the system \eqref{sys_diff} has at least one positive solution $(x,y)\in K$ such that
\[\frac{1}{2}\left(\min_{t \in J}x(t)+\|x\|_\infty\right)>1, \quad \|x\|_\infty<10\]
and
\[\frac{1}{2}\left(\min_{t \in J}y(t)+\|y\|_\infty\right)>\dfrac{1}{4}, \quad \|y\|_\infty<20.\]
\end{example}

\subsection{$\Phi$-Laplace systems}

Consider the problem
\begin{equation}
	\label{phy_lap_sys}
	\begin{cases}
		\left(\Phi_1(x_1^\prime(t))\right)^\prime+f_1(x_1(t),x_2(t))=0, & t \in I,\\
		\left(\Phi_2(x_2^\prime(t))\right)^\prime+f_2(x_1(t),x_2(t))=0, & t \in I,\\
		x_1^\prime(0)=x_1(1)=0=x_2^\prime(0)=x_2(1),
	\end{cases}
\end{equation}
where, for $j\in\{1,2\}$, $\Phi_j:(-a_j,a_j)\rightarrow\mathbb{R}$ are increasing odd homeomorphisms with $0<a_j\leq+\infty$ and $f_j:[0,\infty)\times[0,\infty)\rightarrow[0,\infty)$ are continuous functions. 

Note that each of the homeomorphisms can be \textit{singular} if $a_j<+\infty$ or \textit{classical} if $a_j=+\infty$. The prototype of classical homeomorphism is the $p$-Laplacian operator given by $\Phi:\mathbb{R}\to\mathbb{R}$, $\Phi(x)=\left|x\right|^{p-2}x$, with $p>1$, whereas the paradigmatic example of singular homeomorphism is the mean curvature operator in Minkowski space defined by $\Phi:(-1,1)\to\mathbb{R}$, $\Phi(x)=x(1-x^2)^{-1/2}$.

We seek solutions of system \eqref{phy_lap_sys} in the cone $\mathcal{K} := \mathcal{K}_1 \times \mathcal{K}_2$, where for $j\in\{1,2\}$ we denote
\[\mathcal{K}_j := \{u \in \mathcal{C}(I;\mathbb{R}^+) : 
u \ \text{is concave and nonincreasing}\}.\]

By standard arguments (see \cite{bereanu, herlea}), it follows that solutions of 
$\eqref{phy_lap_sys}$ correspond precisely to the fixed points of the 
completely continuous operator 
$\mathcal{T} = (\mathcal{T}_1,\mathcal{T}_2): \mathcal{K} \to \mathcal{K}$,
defined by 
\[
\mathcal{T}_j(x_1,x_2)(t)
= \int_{t}^{1} 
\Phi_j^{-1}\!\left( \int_{0}^{s} f_j(x_1(\tau), x_2(\tau)) \, d\tau \right) ds,
\qquad t \in I,\quad (j = 1,2).
\]

Operator $\mathcal{T}$ is clearly well-defined, that is, it maps the cone $\mathcal{K}$ into itself. To verify this, fix $x\in\mathcal{K}$ and, for a given $j\in\{1,2\}$, define $g_j:I\to\mathbb{R}$ by
\[g_j(s)=\Phi_j^{-1}\left(\int_{0}^{s}f_j(x_1(\tau),x_2(\tau))dt\right).\]
The nonnegativity of $f_j$ together with the monotonicity of $\Phi_j^{-1}$ ensure that $g_j$ is nondecreasing. It is also immediate that $g_j$ is continuous and nonnegative. Moreover, since
\[\frac{d}{dt}\mathcal{T}_j(x_1,x_2)(t)=-g_j(t),\]
we conclude that $\mathcal{T}_j(x_1,x_2)(t)$ is nonincreasing and concave, as its derivative is negative and nonincreasing.

Let us consider the functional $\gamma_j:\mathcal{K}_j\rightarrow\mathbb{R}^+$ given by
\[\gamma_j(x)=\int_{0}^{1}x(t)dt=\|x\|_{L^1}.\]

Note that the $L^{1}$-norm does not satisfy the usual requirements. For instance, the space $(\mathcal{C}(I;\mathbb{R}^{+}), \|\cdot\|_{1})$ is not complete and as a consequence, the $L^{1}$-norm is not equivalent to the sup-norm $\|\cdot\|_{\infty}$ on $\mathcal{C}(I;\mathbb{R}^{+})$. Nevertheless, we will use $\|\cdot\|_{1}$ as a functional to define a strictly star-shaped set that will allow us to localize a solution.
\begin{remark}
Fix $j \in \{1,2\}$. The functional $\gamma_j$ satisfies conditions $(C_1)$ and $(C_2)$ in Proposition \ref{prop_cond_star}. Indeed, for any $x \in \mathcal{K}_j$, we have $x(t) \ge 0$ for all $t \in I$, and $x$ is concave and nonincreasing. It follows that 
\[x(t) \ge (1-t)x(0) + t x(1) \ge (1-t)x(0) = (1-t)\|x\|_\infty \ \text{ for all } t\in I.\]
Hence, $\gamma_j(x) = \int_0^1 x(t) \, dt \ge \frac{1}{2} \|x\|_\infty$. Therefore, condition $(C_1)$ is fulfilled, while condition $(C_2)$ is an immediate consequence of the linearity of the integral.
\end{remark}

We now present an existence result for the system \eqref{phy_lap_sys}, which follows from applying Theorem~\ref{the_funct}. 

\begin{theorem}
	\label{th_existence_phi}
	Suppose that, for both $j\in \{1,2\}$, there exist $r_j, R_j > 0$ with $2\,r_j < R_j$, such that
\[\Phi_j^{-1}\left(\frac{1}{2} m_j^r \right) > \frac{8}{3} r_j \quad \text{ and } \quad \Phi_j^{-1}\left(M_j^{R}\right) < R_j,\]
	where
	\[
	m_j^r = \min \left\{ f_j(x_1,x_2) : \frac{r_j}{2} \le x_j \le 2\, r_j, \, \frac{r_k}{2} \le x_k \le R_k \, (k\neq j) \right\},
	\]
	\[
	M^{R}_j = \max \left\{ f_j(x_1,x_2) : 0 \le x_1 \le R_1, \, 0 \le x_2 \le R_2 \right\}.
	\]
	
	Then the system \eqref{phy_lap_sys} has at least one positive solution $(x_1,x_2)\in\mathcal{K}$ such that $\gamma_j(x_j)>r_j$ and $\|x_j\|_\infty<R_j$, $(j=1,2)$.
\end{theorem}

\noindent
{\bf Proof.} Firstly, note that for $j\in\{1,2\}$ we have $\overline{\Omega}_{r_j}^{\gamma_j}\subset\Omega_{R_j}^{\|\cdot\|_\infty}$. Take $x\in \left(\overline{\Omega}_{R_1}^{\|\cdot\|_\infty}\setminus\Omega_{r_1}^{\gamma_1}\right)\times\left(\overline{\Omega}_{R_2}^{\|\cdot\|_\infty}\setminus\Omega_{r_2}^{\gamma_2}\right)$ with $\|x_j\|_\infty=R_j$. It is clear that
\[\|\mathcal{T}_j x\|_\infty=\mathcal{T}_j x(0)=\int_{0}^{1}\Phi_j^{-1}\left(\int_{0}^{s}f_j(x(\tau))d\tau\right)ds\leq \Phi_j^{-1}(M_j^R)<R_j=\|x_j\|_\infty.\]
Let $x \in \left(\overline{\Omega}_{R_1}^{\|\cdot\|_\infty} \setminus \Omega_{r_1}^{\gamma_1}\right) 
\times \left(\overline{\Omega}_{R_2}^{\|\cdot\|_\infty} \setminus \Omega_{r_2}^{\gamma_2}\right)$ 
with $\gamma_j(x_j) = r_j$. Then $\|x_j\|_\infty = x_j(0) \le 2\, r_j$. Since $x_j$ is concave and nonnegative, we have 
\[
x_j\left(\frac{1}{2}\right) \ge \frac{x_j(0)}{2} \ge \frac{r_j}{2}.
\] 
Moreover, as $x_j$ is nonincreasing, it follows that 
\[
\frac{r_j}{2} \le x_j(\tau) \le 2\, r_j \quad \text{for all } \tau \in \left[0, \frac{1}{2}\right].
\] 
Similarly, for $k \neq j$, since $x_k \in \overline{\Omega}_{R_k}^{\|\cdot\|_\infty} \setminus \Omega_{r_k}^{\gamma_k}$, we obtain
\[
\frac{r_k}{2} \le x_k(\tau) \le R_k \quad \text{for all } \tau \in \left[0, \frac{1}{2}\right].
\] 

Hence, we obtain
\begin{equation*}
	\begin{aligned}
		\mathcal{T}_j x\left(\frac{1}{2}\right) 
		&= \int_{\frac{1}{2}}^{1} \Phi_j^{-1}\Biggl(\int_0^s f_j(x_1(\tau),x_2(\tau)) \, d\tau\Biggr) ds \ge \int_{\frac{1}{2}}^{1} \Phi_j^{-1}\Biggl(\int_0^{\frac{1}{2}} f_j(x_1(\tau),x_2(\tau)) \, d\tau\Biggr) ds \\
		&= \frac{1}{2} \Phi_j^{-1}\Biggl(\int_0^{\frac{1}{2}} f_j(x_1(\tau),x_2(\tau)) \, d\tau\Biggr) \ge \frac{1}{2} \Phi_j^{-1}\left(\frac{1}{2} m_j^r\right) > \frac{4}{3} r_j.
	\end{aligned}
\end{equation*}

Finally, using the concavity and nonincreasing property of $\mathcal{T}_j x$, we have
\[
\gamma_j(\mathcal{T}_j x) = \int_0^{\frac{1}{2}} \mathcal{T}_j x(t) \, dt + \int_{\frac{1}{2}}^1 \mathcal{T}_j x(t) \, dt 
\ge \frac{3}{4} \mathcal{T}_j x\left(\frac{1}{2}\right) > r_j = \gamma_j(x_j),
\]
which establishes the last desired inequality for applying Theorem \ref{the_funct}. As a conclusion, problem \eqref{phy_lap_sys} has at least one solution $(x_1,x_2)\in\mathcal{K}$ such that $\gamma_j(x_j)=\int_{0}^{1}x_j(t)dt>r_j $ and $\|x_j\|_\infty<R_j$ ($j=1,2$).\qed

\begin{remark}
	\label{remark_singular}
	If $\Phi_j$ is singular, then condition $\Phi_j^{-1}(M_j^R)<R_j$ is trivially satisfied for $R_j:=a_j$.
\end{remark}

We may consider an asymptotic behavior on the quotient $f_j/\Phi_j$ at zero, which preserves the existence of a solution for \eqref{phy_lap_sys} with both nontrivial components. Accordingly, we assume that, for each $j \in \{1,2\}$, homeomorphism $\Phi_j$ is singular and satisfies
\begin{equation}
	\label{eq_asymp_phi}
\limsup_{x \to 0^+} \frac{\Phi_j(\tau x)}{\Phi_j(x)} < +\infty \quad \text{for all } \tau>1.
\end{equation}
The technical assumption \eqref{eq_asymp_phi} is commonly found in the existing literature. For instance, it has been invoked in \cite{DrGaMa} in the context of a \textit{classical} homeomorphism and in \cite{bereanu1} for a \textit{singular} $\Phi$-Laplacian.

\begin{corollary}
	\label{cor_final}
	Suppose that, for both $j\in\{1,2\}$, $\Phi_j$ is singular, satisfies \eqref{eq_asymp_phi} and the function $f_j$ is nondecreasing and 
	\begin{equation}
		\label{eq_asymp_f}
		\lim_{x\rightarrow0^+}\frac{f_j(x,x)}{\Phi_j(x)}=+\infty.
	\end{equation}
	
	Then \eqref{phy_lap_sys} has at least one positive solution with both nontrivial, concave and nonincreasing components.
\end{corollary}
\noindent
{\bf Proof.} Note that, for a fixed $j \in \{1,2\}$, if $f_j$ is nondecreasing in both arguments and $r_1 = r_2 = r$, then $m_j^r = f_j\left(\frac{r}{2}, \frac{r}{2}\right).$ Hence, inequality $\Phi_j^{-1}\left(\frac{1}{2} m_j^r\right) > \frac{8}{3} r_j$, in Theorem \ref{th_existence_phi}, can be equivalently rewritten as
\[
2 < \frac{f_j(r,r)}{\Phi_j\left(\frac{16}{3} r\right)} 
= \frac{f_j(r,r)}{\Phi_j(r)} \cdot \frac{\Phi_j(r)}{\Phi_j\left(\frac{16}{3} r\right)}.
\]
By \eqref{eq_asymp_phi} and \eqref{eq_asymp_f}, it follows that there exists a sufficiently small $r \in \mathbb{R}_+$ for which this holds.  
The result then follows directly from Theorem~\ref{th_existence_phi}, in view of Remark~\ref{remark_singular}.
 \qed

Finally, the existence theory is illustrated with the following example, which involves power-type nonlinearities. 
 
 \begin{example}
 Our Corollary \ref{cor_final} guarantees that there exists at least one positive solution with both components nontrivial for a system of the form
 \begin{equation*}
% 	\label{phy_lap_sys}
 	\begin{cases}
 		-\left(\frac{x^\prime}{\sqrt{1-(x^\prime)^2}}\right)^\prime=a\, x^\alpha y^\beta,\\
 		-\left(\frac{y^\prime}{\sqrt{1-(y^\prime)^2}}\right)^\prime=b\,x^\lambda+c\,y^\mu,\\
 		x^\prime(0)=x(1)=0=y^\prime(0)=y(1),
 	\end{cases}
 \end{equation*}
 provided that $a,b,c>0$, $\alpha,\beta,\lambda,\mu\geq 0$, $\alpha+\beta<1$, $\lambda<1$ and $\mu<1$.
 
 In this case, the differential operator is given by the \textit{mean curvature operator in Minkowski space}, i.e., the singular homeomorphism $\Phi_j:(-1,1)\to\mathbb{R}$, $\Phi_j(x)=x(1-x^2)^{-1/2}$, $j=1,2$, which clearly satisfies \eqref{eq_asymp_phi}.
 On the other hand, observe that the system admits the trivial solution $(x,y)\equiv(0,0)$ as well as a semi-trivial solution $(0,\bar{y})$. The latter can be obtained as an application of \cite[Theorem 1]{bereanu1} to the single equation
 \[-\left(\frac{y^\prime}{\sqrt{1-(y^\prime)^2}}\right)^\prime=c\,y^\mu, \quad y^\prime(0)=y(1)=0. \]
 \end{example}

\section*{Acknowledgements}

L. M. Fern\'andez-Pardo acknowledges the financial support of the Spanish Ministry of Science, Innovation and Universities (Grant reference FPU24/01545).
J. Rodr\'iguez--L\'opez has been partially supported by Ministerio de Ciencia y Tecnología (Spain), AEI and Feder, grant PID2020-113275GB-I00, and Xunta de Galicia, Spain, Project ED431C 2023/12.

\section*{Author Contributions}

All authors contributed equally to the preparation of this manuscript.

\section*{Competing interests}

The authors have no relevant financial or non-financial interests to disclose.

\section*{Data availability}

The manuscript has no associated data.

\end{document}